\def\draft{y}
\documentclass[12pt]{amsart}
\pdfoutput=1
\usepackage[headings]{fullpage}
\usepackage{amssymb,epic,eepic,epsfig}
\usepackage{graphicx}
\usepackage{texdraw}
\usepackage{url}
\usepackage{subfigure}
\usepackage{multirow}
\usepackage{array}
\usepackage[bookmarks=true,%
    colorlinks=true,%
    linkcolor=blue,%
    citecolor=blue,%
    filecolor=blue,%
    menucolor=blue,%
    urlcolor=blue,%
    breaklinks=true]{hyperref}

\usepackage{enumerate}
  
\usepackage{mathdots}
\usepackage{amsfonts}

\usepackage{mathtools,amssymb}

\usepackage{wasysym}

\usepackage{caption}

\usepackage[all]{xy}

\newtheorem{theorem}{Theorem}[section]
\theoremstyle{definition}
\newtheorem{proposition}[theorem]{Proposition}
\newtheorem{lemma}[theorem]{Lemma}
\newtheorem{definition}[theorem]{Definition}
\newtheorem{remark}[theorem]{Remark}
\newtheorem{corollary}[theorem]{Corollary}

\newtheorem{convention}[theorem]{Convention}

\numberwithin{equation}{section}

\def\printname#1{
        \if\draft y
                \smash{\makebox[0pt]{\hspace{-0.5in}
                        \raisebox{8pt}{\tt\tiny #1}}}
        \fi
}

\newlength{\standardunitlength}
\catcode`\@=11
\long\def\@makecaption#1#2{%
     \vskip 10pt

\setbox\@tempboxa\hbox{
       \small\sf{\bfcaptionfont #1. }\ignorespaces #2}%
     \ifdim \wd\@tempboxa >\captionwidth {%
         \rightskip=\@captionmargin\leftskip=\@captionmargin
         \unhbox\@tempboxa\par}%
       \else
         \hbox to\hsize{\hfil\box\@tempboxa\hfil}%
     \fi}
\font\bfcaptionfont=cmssbx10 scaled \magstephalf
\newdimen\@captionmargin\@captionmargin=2\parindent
\newdimen\captionwidth\captionwidth=\hsize
\catcode`\@=12

\def\lbl#1{\label{#1}\printname{#1}}

\def\cxymatrix#1{\xy*[c]\xybox{\xymatrix#1}\endxy}

\def\BZ{\mathbb Z}

\def\BC{\mathbb C}

\def\a{\alpha}

\def\t{\tau}
\def\e{\epsilon}

\def\b{\beta}

\def\longto{\longrightarrow}

\def\SL{\mathrm{SL}}
\def\PSL{\mathrm{PSL}}
\def\PGL{\mathrm{PGL}}

\def\pt{\partial}

\def\ep{\epsilon}

\DeclareMathOperator{\id}{id}

\DeclareMathOperator{\Ext}{Ext}
\DeclareMathOperator{\rank}{rank}
\DeclareMathOperator{\ori}{ori}
\DeclareMathOperator{\rest}{rest}
\DeclareMathOperator{\cusp}{cusp}

\def\be{  \begin{equation} }
\def\ee{  \end{equation} }

\def\A{\mathcal A}

\def\Hom{\mathrm{Hom}}

\def\diag{\text{diag}}

\def\CT{\mathcal T}

\def\R{\mathbb{R}}
\def\Z{\mathbb{Z}}
\def\C{\mathbb{C}}

\def\T{\mathcal{T}}
\def\J{\mathcal{J}}

\def\Ker{\mathrm{Ker}}
\def\Im{\mathrm{Im}}

\def\fg{\mathfrak{g}}

\usepackage{ifpdf}
\ifpdf
\fi

\begin{document}

\begin{abstract}
In \cite{GaroufalidisGoernerZickert} we studied 
$\mathrm{PGL}(n,\C)$-representations of a 3-manifold via a generalization of 
Thurston's gluing equations. Neumann has proved some symplectic properties 
of Thurston's gluing equations that play an important role in recent 
developments of exact and perturbative Chern-Simons theory. In this paper, 
we prove the symplectic properties of the $\mathrm{PGL}(n,\C)$-gluing equations 
for all ideal triangulations of compact oriented 3-manifolds.
\end{abstract}


\title[The symplectic properties of the $\PGL(n,\BC)$-gluing equations]{
The symplectic properties of the $\PGL(n,\BC)$-gluing equations}
\author{Stavros Garoufalidis}
\address{School of Mathematics \\
         Georgia Institute of Technology \\
         Atlanta, GA 30332-0160, USA \newline
         {\tt \url{http://www.math.gatech.edu/~stavros }}}
\email{stavros@math.gatech.edu}
\author{Christian K. Zickert}
\address{University of Maryland \\
         Department of Mathematics \\
         College Park, MD 20742-4015, USA \newline
         {\tt \url{http://www2.math.umd.edu/~zickert}}}
\email{zickert@umd.edu}
\thanks{
The authors were supported in part by the National Science Foundation. \\
\newline
1991 {\em Mathematics Classification.} Primary 57N10. Secondary 57M27.
}

\date{October 8, 2013}

\maketitle

\tableofcontents


\section{Introduction}
\label{sec.ge}

Thurston's gluing equations are a system of polynomial equations that were 
introduced to concretely construct hyperbolic structures. They are defined 
for every compact, oriented $3$-manifold $M$ with arbitrary, posibly empty, 
boundary together with a topological ideal triangulation $\T$. The system has 
the form 

\begin{equation}
\label{eq:GluingEqIntro}
\prod_j z_j^{A_{ij}} \prod_j (1-z_j)^{B_{ij}}=\e_i,
\end{equation}
where $A$ and $B$ are matrices whose columns are parametrized by the 
simplices of $\T$ and $\e_i$ is a sign.
Each (non-degenerate) solution explicitly determines (up to conjugation) a 
representation of $\pi_1(M)$ in $\PGL(2,\C)$.

The matrices $A$ and $B$ in~\eqref{eq:GluingEqIntro} have some remarkable 
symplectic properties that play a fundamental role in exact and perturbative 
Chern-Simons theory for 
$\SL(2,\C)$~\cite{DGLZ,Di1,DGG2,DG,Ga-3Dindex,GHRS,DGon}.

In~\cite{GaroufalidisGoernerZickert} Garoufalidis, Goerner and Zickert 
generalized Thurston's gluing equations to representations in $\PGL(n,\C)$, 
i.e.~they constructed a system of the form~\eqref{eq:GluingEqIntro} such that 
each solution determines a representation of $\pi_1(M)$ in $\PGL(n,\C)$.
The $\PGL(n,\C)$-gluing equations are expected to play a similar role in 
$\PGL(n,\C)$-Chern-Simons theory as Thurston's gluing equations play in 
$\PSL(2,\C)$-Chern-Simons theory.

In this paper we focus on the symplectic properties of the $\PGL(n,\C)$-gluing 
equations. This was initiated in~\cite{GaroufalidisGoernerZickert}, where we 
proved that the rows of $(A\vert B)$ are symplectically orthogonal.
The symplectic properties for $n=2$ play a key role in the definition of the
formal power series invariants of \cite{DG} (conjectured to be asymptotic
to all orders of the Kashaev invariant) and in the definition of the 3D-index
of Dimofte--Gaiotto--Gukov \cite{DGG2} whose convergence and topological
invariance was established in \cite{Ga-3Dindex} and \cite{GHRS}. Our results
fulfill a wish of the physics literature \cite{DGon}, and may be used for
an extension of the work \cite{DG,GHRS,Bonahon} to the setting of the 
$\PGL(n,\C)$-representations. 


\section{Preliminaries and statement of results}

\subsection{Triangulations}
Let $M$ be a compact, oriented $3$-manifold with (possibly empty) boundary, 
and let $\widehat M$ be the space obtained from $M$ by collapsing each 
boundary component to a point.

\begin{definition}
A (topologically ideal) \emph{triangulation} of $M$ is an identification of 
$M$ with a space obtained from a collection of simplices by gluing together 
pairs of faces via affine homeomorphisms. A \emph{concrete triangulation} is 
a triangulation together with an identification of each simplex of $M$ with a 
standard ordered $3$-simplex. A concrete triangulation is \emph{oriented} if 
for each simplex, the orientation induced by the identification with a 
standard simplex agrees with the orientation of $M$.
\end{definition}

\begin{remark}
Unless otherwise specified, a triangulation always refers to an 
\emph{oriented} triangulation. The census triangulations are concrete 
triangulations, which are oriented when $M$ is orientable. All of our 
results can be generalized to arbitrary concrete triangulations (e.g.~ordered 
triangulations) by introducing additional signs. We will not pursue this here.
\end{remark}

\subsection{Thurston's gluing equations}

We briefly review Thurston's gluing equations. For details, see 
Thurston~\cite{ThurstonNotes} or Neumann--Zagier~\cite{NeumannZagier}.
Let $z_j\in\C\setminus\{0,1\}$ be complex numbers, one for each simplex 
$\Delta_j$ of $\T$. Assign \emph{shape parameters} $z_j$, 
$z_j^{\prime}=\frac{1}{1-z_j}$ and $z_j^{\prime\prime}=1-\frac{1}{z}$ to the edges 
of $\Delta_j$ as in Figure~\ref{fig:ShapeParameters}.

\begin{figure}[htb]
\begin{center}
\begin{minipage}[b]{0.47\textwidth}
\begin{center}
\scalebox{0.8}{\input{ShapeParameters.tex}}
\end{center}
\end{minipage}
\hfill
\begin{minipage}[b]{0.47\textwidth}
\begin{center}
\scalebox{0.82}{\input{SL2Quiver.tex}}
\end{center}
\end{minipage}\\
\begin{minipage}[t]{0.47\textwidth}
\begin{center}
\caption{Shape parameters.}
\label{fig:ShapeParameters}
\end{center}
\end{minipage}
\hfill
\begin{minipage}[t]{0.47\textwidth}
\begin{center}
\caption{Quiver representation of $\Omega$.}
\label{fig:SL2Quiver}
\end{center}
\end{minipage}
\end{center}
\end{figure}

\subsubsection{Edge equations} 

We have a gluing equation for each $1$-cell $e$ of $\T$ defined by setting 
equal to $1$ the product of all shape parameters assigned to the edges of $e$. 
The gluing equation for $e$ can thus be written in the form
\begin{equation}
\label{eq:GluingEqForm}
\prod_{j}(z_j)^{A^\prime_{e,j}}\prod_j (z_j^{\prime})^{B^{\prime}_{e,j}}
\prod_j(z_j^{\prime\prime})^{C^\prime_{e,j}}=1,\quad \text{or}\quad 
\prod_{j}(z_j)^{A_{e,j}}\prod_j (1-z_j)^{B_{e,j}}=\varepsilon_e
\end{equation}
where $A=A^\prime-C^\prime$ and $B=C^\prime-B^\prime$ are the so-called gluing 
equation matrices. Each solution determines a representation 
$\pi_1(M)\to\PGL(2,\C)$.
Note that the rows of the gluing equation matrices are parametrized by 
$1$-cells, and the columns by the simplices of $\T$.

\subsubsection{Cusp equations} 

Each solution $z=\{z_j\}$ to the edge equations gives rise to a cohomology 
class $C(z)\in H^1(\partial M;\C^*)$ obtained by taking the product of the 
shape parameters along normal curves. This class vanishes if and only if the 
representation corresponding to $z$ is boundary-unipotent. The vanishing of 
$C(z)$ is equivalent to a system of equations 
\begin{equation}
\prod_{j}(z_j)^{A^{\cusp\prime}_{\alpha,j}}
\prod_j (z_j^{\prime})^{B^{\cusp\prime}_{\alpha,j}}
\prod_j(z_j^{\prime\prime})^{C^{\cusp\prime}_{\alpha,j}}=1,
\quad \text{or}\quad \prod_{j}(z_j)^{A^{\cusp}_{\alpha,j}}
\prod_j (1-z_j)^{B^{\cusp}_{\alpha,j}}=\varepsilon_\alpha
\end{equation}
of the form~\eqref{eq:GluingEqForm} with an equation for each generator 
$\alpha$ of $H_1(\partial M)$.
These equations are called \emph{cusp equations}.
Note that the rows of the cusp equation matrices are parametrized by 
generators $\alpha$ of $H_1(\partial M)$, and the columns by the simplices 
of $\T$.

\subsection{Neumann's chain complex}
\lbl{sub.Nchain}

For an ordered $3$-simplex $\Delta$, let $J_\Delta$ denote the free abelian 
group generated by the \emph{unoriented} edges of $\Delta$ subject to the 
relations
\begin{gather}
\varepsilon_{01}=\varepsilon_{23},\qquad \varepsilon_{12}=\varepsilon_{03},
\qquad \varepsilon_{02}=\varepsilon_{13}\label{eq:first}\\
\varepsilon_{01}+\varepsilon_{12}+\varepsilon_{02}=0.\label{eq:second}
\end{gather}
Here $\varepsilon_{ij}$ denotes the edge between vertices $i$ and $j$ of 
$\Delta$. Note that~\eqref{eq:first} states that two opposite edges are 
equal, and that~\eqref{eq:first} and~\eqref{eq:second} together imply that 
the sum of the edges incident to a vertex is $0$.

The space $J_\Delta$ is endowed with a non-degenerate skew symmetric bilinear 
form $\Omega$ defined uniquely by
\begin{align}
\Omega(\varepsilon_{01},\varepsilon_{12})&
=\Omega(\varepsilon_{12},\varepsilon_{02})=
\Omega(\varepsilon_{02},\varepsilon_{01})=1.
\end{align}
The form $\Omega$ may be represented by the quiver in 
Figure~\ref{fig:SL2Quiver}. Namely, each edge corresponds to a vertex of 
the quiver, and $\Omega(\varepsilon,\varepsilon^{\prime})=1$ if and only if 
there is a directed edge in the quiver going from $\varepsilon$ to 
$\varepsilon^{\prime}$. 

Neumann~\cite[Thm~4.1]{NeumannComb} encoded the symplectic 
properties of the gluing equations in terms of a chain complex
$\J=\J(\CT)$ (indexed such that $\J_5=\J_1=C_0(\T)$) 
\be
\label{eq:2chain}
\xymatrix{0 \ar[r] & C_0(\CT) \ar[r]^\alpha & C_1(\CT)\ar[r]^\beta
& J(\T)\ar[r]^{\beta^*} & C_1(\CT) \ar[r]^{\alpha^*} & C_0(\CT) \ar[r] &0}
\ee
defined combinatorially from the triangulation $\CT$. Here
\begin{itemize}
\item
$C_i(\CT)$ is the free $\BZ$-module of the unoriented $i$-simplices
of $\CT$.
\item
$J(\CT)=\bigoplus_{\Delta\in\T}J_{\Delta}$, with $\Omega$ extended orthogonally. 
\item$\alpha$ takes a $0$-cell to the sum of incident $1$-cells.
\item $\beta$ takes a $1$-cell to the sum of its edges.
\item $\a^*$ maps an edge to the sum of its endpoints.
\item $\b^*$ is the unique rotation preserving map taking $\varepsilon_{01}$ 
to $[\varepsilon_{03}]+[\varepsilon_{12}]-[\varepsilon_{02}]-[\varepsilon_{13}]$.
\end{itemize}
Since $\beta^*\circ\beta=0$, $\Ker(\beta^*)$ is orthogonal to $\Im(\beta)$, 
so $\Omega$ descends to a form on $H_3(\J)$, which is non-degenerate modulo 
torsion.


\begin{theorem}
[{Neumann~\cite[Thm~4.2]{NeumannComb}}]
\label{thm:Neumann}
The homology groups of $\J$ are given by
\begin{equation}
\begin{gathered}
H_5(\J)=0,\qquad H_4(\J)=\BZ/2\Z,\qquad H_3(\J)= 
K \oplus H^1(\widehat M;\BZ/2\Z)
\\
H_2(\J)=H_1(\widehat M;\BZ/2\Z),\qquad H_1(\J)=\Z/2\Z,
\end{gathered}
\end{equation}\label{eq:IsoTensorHalf}
where $K=\mathrm{Ker}\big(H_1(\pt M,\BZ) \longto H_1(M,\BZ/2\Z)\big)$. 
Moreover, the
isomorphism 
\begin{equation}
H_3(\J)/\mathrm{torsion} \cong K
\end{equation}
identifies $\Omega$ with the intersection form on $H_1(\partial M)$.
\end{theorem}

\begin{remark}
\label{rm:TwiceIntersection}
Under the isomorphism
\begin{equation}
\label{eq:TwiceIntersection}
H_3(\J)\otimes \Z[1/2] \cong H_1(\partial M;\Z[1/2]),
\end{equation}
the form $\Omega$ corresponds to \emph{twice} the intersection form.
\end{remark}

\subsection{Symplectic properties of the gluing equations}
\label{sub:properties.n=2}
Neumann's result implies some important symplectic properties of the gluing 
equation matrices. We formulate them here in a way that generalizes
to the $\PGL(n,\BC)$ setting.

By the definition of $\beta$ we have for each $1$-cell $e$ an element
\begin{equation}
\label{eq:Betae}
\beta(e)=\sum_jA^\prime_{e,j}\varepsilon_{01,j}
+\sum_jB^\prime_{e,j}\varepsilon_{12,j}
+\sum_jC^\prime_{e,j}\varepsilon_{02,j}=
\sum_jA_{e,j}\varepsilon_{01,j}+\sum_jB_{e,j}\varepsilon_{12,j}
\end{equation}
in $\Im(\beta)$. Similarly, for a generator $\alpha$ of $H_1(\partial M)$, 
we have the element
\begin{equation}
\label{eq:DeltaAlpha}
\delta(\alpha)=\sum_jA^\prime_{e,j}\varepsilon_{01,j}
+\sum_jB^\prime_{e,j}\varepsilon_{12,j}+\sum_jC^\prime_{e,j}\varepsilon_{02,j}=
\sum_jA_{e,j}\varepsilon_{01,j}+\sum_jB_{e,j}\varepsilon_{12,j},
\end{equation}
which Neumann shows is in $\Ker(\beta^*)$.

\begin{corollary}
\label{cor:PoissonCommute}
Let $w_J$ be the standard symplectic form on $\Z^{2t}$ given by 
$J=\left(\begin{smallmatrix} 0 & I \\ -I & 0 \end{smallmatrix}\right)$, 
where $t$ is the number of simplices of $\CT$
and let $\iota$ denote the intersection form on $H_1(\partial M)$. 
\begin{enumerate}[(i)]
\item For any rows $x$ and $y$ of $(A\vert B)$, $w(x,y)=0$.
\item For any rows $x$ of $(A\vert B)$ and $y$ of $(A^{\cusp}\vert B^{\cusp})$, 
$w(x,y)=0$.
\item For any rows $x$ and $y$ of $(A^{\cusp}\vert B^{\cusp})$ corresponding 
to $\alpha$ and $\beta$ in $H_1(M,\partial M)$, respectively, 
$w(x,y)=\Omega(\delta(\alpha),\delta(\beta))=2\iota(\alpha,\beta)$.
\end{enumerate}
\end{corollary}

\begin{proof}
The first and second statement follow from the fact that 
$\beta^*\circ\beta=0$, which implies that $\Ker(\beta^*)$ is symplectically 
orthogonal to $\Im(\beta)$. The third result is proved in 
Neumann~\cite{NeumannComb}, c.f.~Remark~\ref{rm:TwiceIntersection}. Namely 
$\delta\colon H_1(\partial M)\to H_3(\J)$ induces the isomorphism 
in~\eqref{eq:TwiceIntersection}.
\end{proof}

\begin{corollary}
\label{cor:Rank} 
The rank of $(A\vert B)$ is the number of edges minus the number of cusps.
\end{corollary}

\begin{proof}
This follows from the fact that $H_4(\J)$ is zero modulo torsion.
\end{proof}

\begin{remark}
A simple argument that uses the Euler characteristic shows that the number of 
edges of $\CT$ equals $t+c-h$, where $t$ is the number of simplices, 
$h=\frac{1}{2}\rank(H_1(\partial M))$ and $c$ is the number of boundary 
components. Hence, the matrix matrix $(A\vert B)$ has dimension 
$(t+c-h)\times2t$. In particular, if all boundary components are tori 
(the case of most interest), the dimension is $t\times 2t$. If we extend a 
basis for the row span of $(A\vert B)$ by cusp equations, the resulting 
$t\times 2t$ matrix has full rank, and is the upper half of a symplectic 
matrix. Such matrices play a crucial role in \cite{Di1,DG,DGG1,DGG2,GHRS}.
\end{remark}

\subsection{Statement of results}
\label{sub:results}

As described in~\cite{GaroufalidisGoernerZickert} (and reviewed in 
Section~\ref{sec:GluingEqsReview} below), there is a $\PGL(n,C)$-gluing 
equation for each non-vertex integral point $p$ of $\T$. The gluing equation 
for $p$ can be written as
\begin{equation}
\label{eq:GeneralGluingEqForm}
\prod_{(s,\Delta)}(z_{s,\Delta})^{A_{p,(s,\Delta)}}
\prod_{(s,\Delta)} (1-z_{s,\Delta})^{B_{p,(s,\Delta)}}=\varepsilon_p
\end{equation}
for matrices $A$ and $B$ whose rows are parametrized by the (non-vertex) 
integral points of $\T$ and columns by the set of subsimplices of the 
simplices of $\T$.

Furthermore there is a cusp equation for each generator $\alpha\otimes e_r$ 
of $H_1(\partial M;\Z^{n-1})$ of the form
\begin{equation}
\label{eq:GeneralCuspEqForm}
\prod_{(s,\Delta)}(z_{s,\Delta})^{A^{\cusp}_{\alpha\otimes e_r,(s,\Delta)}}
\prod_{(s,\Delta)} (1-z_{s,\Delta})^{B^{\cusp}_{\alpha\otimes e_r,(s,\Delta)}}
=\varepsilon_{\alpha\otimes e_r}
\end{equation}
for matrices $A^{\cusp}$ and $B^{\cusp}$ whose rows are parametrized by 
generators $\alpha\otimes e_r$ of $H_1(\partial M;\Z^{n-1})$ and columns by 
the set of subsimplices of the simplices of $\T$.
 
In Section~\ref{sec:ChainComplexDefn} below we define a chain complex
$\J^{\fg}=\J^{\fg}(\CT)$ (indexed such that $\J^{\fg}_5=\J^{\fg}_1=C^{\fg}_0(\T)$)
\begin{equation}\label{eq:Nchain}
\xymatrix{0 \ar[r] & C^{\fg}_0(\CT) \ar[r]^\alpha & C^{\fg}_1(\CT)\ar[r]^\beta
& J^{\fg}(\T)\ar[r]^{\beta^*} & C^{\fg}_1(\CT) \ar[r]^{\alpha^*} & 
C^{\fg}_0(\CT) \ar[r] &0}
\end{equation}
generalizing \eqref{eq:2chain}. Here $\fg$ denotes the Lie algebra of 
$\SL(n,\C)$, the notation being in anticipation of a generalization to 
arbitrary simple, complex Lie algebras. The three middle terms of $\J^{\fg}$ 
appeared already in 
Garoufalidis--Goerner--Zickert~\cite{GaroufalidisGoernerZickert}.
There is a non-degenerate antisymmetric form on $J^{\fg}(\T)$ descending to 
a non-degenerate form on $H_3(\J^{\fg})$ modulo torsion.
\begin{theorem}\label{thm:1} Let $h=\frac{1}{2}\rank(H_1(\partial M))$. The 
homology groups of $\J^{\fg}$ are given by
\begin{equation}
\begin{gathered}
H_5(\J^{\fg})=0,\qquad H_4(\J^{\fg})=\BZ/n\Z,\qquad 
H_3(\J^{\fg})= K \oplus H^1(\widehat M;\BZ/n\Z)
\\
H_2(\J^{\fg})=H_1(\widehat M;\BZ/n\Z),\qquad H_1(\J^{\fg})=\Z/n\Z,
\end{gathered}
\end{equation}
where $K\subset H_1(\pt M,\Z^{n-1})$ is a subgroup of index $n^h$.
Moreover, the isomorphism 
\begin{equation}\label{eq:IsoModN}
H_3(\J^{\fg})\otimes\Z[1/n]\cong H_1(\partial M;\Z[1/n]^{n-1})
\end{equation}
identifies $\Omega$ with the non-degenerate form $\omega_{A_{\fg}}$ on 
$H_1(\partial M;\Z[1/n]^{n-1})$ given by
\begin{equation}
\label{eq:OmegaA}
\omega_{A_{\fg}}(\alpha\otimes v,\beta\otimes w)=
\iota(\alpha,\beta)\langle v,A_{\fg}w\rangle,
\end{equation}
where $\iota$ is the intersection form and $A^{\fg}$ is the Cartan matrix of 
$\fg$.
\qed
\end{theorem} 

\begin{remark}
Presumably, $K=\Ker\big(H_1(\partial M;\Z^{n-1})\to H_1(M;\Z/n\Z)\big)$, 
where $\Z/n\Z$ is regarded as the quotient of  $\Z^{n-1}$ by the Cartan matrix.
This would be analogous to Neumann's result Theorem~\ref{thm:Neumann}. 
\end{remark}

For each integral point $p$, the element
\begin{equation}
\label{eq:BetapGeneral}
\sum_{(s,\Delta)} A_{p,(s,\Delta)}(s,\varepsilon_{01})_\Delta
+\sum_{s,\Delta} B_{p,(s,\Delta)}(s,\varepsilon_{12})_\Delta
\end{equation}
equals $\beta(p)$, and is thus in $\Im(\beta)$. 
 For each generator $\alpha\otimes e_r$ of $H_1(\partial M;\Z^{n-1})$ we show 
that the element
\begin{equation}
\sum_{(s,\Delta)}A^{\cusp}_{\alpha\otimes e_r,(s,\Delta)}(s,\varepsilon_{01})_\Delta
+\sum_{(s,\Delta)}B^{\cusp}_{\alpha\otimes e_r,(s,\Delta)}(s,\varepsilon_{12})_\Delta,
\end{equation}
is in the kernel of $\beta^*$. In fact it equals 
$\delta^{\prime}(\alpha\otimes e_r)$ for a map $\delta^{\prime}\colon 
H_1(\partial M;\Z^{n-1})$ which induces the isomorphism~\eqref{eq:IsoModN}.

Analogously to Corollary~\ref{cor:PoissonCommute} we have.

\begin{corollary}
\label{cor:PoissonCommuteGeneral}
The rows of $(A\vert B)$ and $(A^{\cusp}\vert B^{\cusp})$ are orthogonal with 
respect to the standard symplectic form $\omega_J$. Moreover, if $x$ and $y$ 
are rows of $(A^{\cusp}\vert B^{\cusp})$ corresponding to $\alpha\otimes e_r$ 
and $\beta\otimes e_s$, respectively, we have
\begin{equation}
\omega_J(x,y)=\Omega\big(\delta^{\prime}(\alpha\otimes e_r),
\delta^{\prime}(\beta\otimes e_s)\big)
=\iota(\alpha,\beta)\langle e_r,A_{\fg}e_s\rangle.\qed
\end{equation}
\end{corollary}

\begin{corollary}
\label{cor:RankGeneral} 
The rank of $(A\vert B)$ is the number of non-vertex integral points minus 
$c(n-1)$, where $c$ is the number of boundary components.
\qed
\end{corollary}

\begin{remark} 
If all boundary components are tori, $(A\vert B)$ has twice as many columns 
as rows, and $c(n-1)=\frac{1}{2}\rank H_1(\partial M;\Z^{n-1})$. It follows 
that one can extend a basis for the row space of $(A\vert B)$ by adding 
cusp equations to obtain a matrix with full rank. This matrix is the upper 
part of a symplectic matrix and as stated in the introduction plays a 
crucial role in extending the work of \cite{Di1,DG,DGG1,DGG2,GHRS} to the
$\PGL(n,\BC)$ setting. This will be discussed in a future publication.
\end{remark}

\subsection{A side comment on quivers}

If you take a quiver as in Figure~\ref{fig:SL2Quiver} for each subsimplex 
and superimpose them canceling edges with opposite orientations, you get the 
quiver shown in Figure~\ref{f.pgln.quiver}. Everything cancels in the 
interior. The quiver on the face equals the quiver in 
Fock--Goncharov~\cite[Fig.~1.5]{FockGoncharov}, and also appears for $n=3$ 
in Bergeron--Falbel--Guilloux~\cite[Fig.~4]{BFG}. One can go from 
the quiver on two of the faces to the quiver on the two other faces by 
performing quiver mutations (see e.g.~Keller~\cite{Keller}). The quiver 
mutations change the $X$-coordinates and Ptolemy coordinates by cluster 
mutations~\cite{ClusterAlgebras3}, and there is a one-one correspondence 
between quiver mutations and subsimplices. Although we do not need any of 
this here, this observation was a major motivation for 
\cite{GaroufalidisThurstonZickert} and \cite{GaroufalidisGoernerZickert}.

\begin{figure}[htpb]
\input{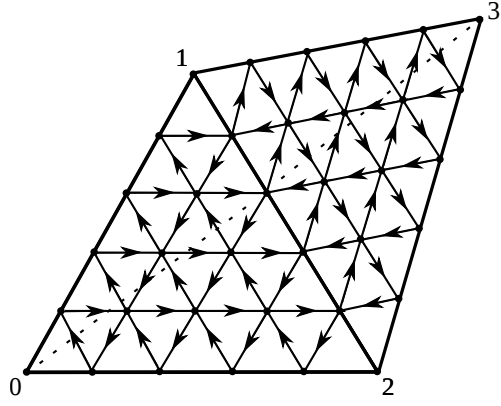}
\caption{Superposition of copies of the quiver in Figure~\ref{fig:SL2Quiver}, 
one for each subsimplex.}
\label{f.pgln.quiver}
\end{figure}


\section{Shape assignments and gluing equations}
\label{sec:GluingEqsReview}

Fix a manifold $M$ with a triangulation $\T$. We identify each simplex of 
$M$ with the simplex
\begin{equation}
\Delta^3_n = \big\{(x_0,x_1,x_2,x_3) \in \R^4\bigm\vert 0\leq x_i \leq n, 
x_0+x_1+x_2+x_3 = n\big\}.
\end{equation}
Let $\Delta_n^3(\Z)$ and denote the integral points of $\Delta^3_n$, and 
$\dot\Delta^3_n(\Z)$ denote the integral points with the $4$ vertex points 
removed.
The natural left $A_4$-action on $\Delta^3_n$ given by 
\begin{equation}
\sigma(x_0,\dots,x_3)=(x_{\sigma^{-1}(0)},\dots,x_{\sigma^{-1}(3)})
\end{equation}
induces $A_4$-actions on $\Delta_n^3(\Z)$ and $\dot\Delta^3_n(\Z)$ as well.

\begin{definition}
\label{def:Subsimplex}
A \emph{subsimplex} of $\Delta^3_n$ is a subset $S$ of $\Delta^3_n$ obtained 
by translating $\Delta^3_2\subset\R^4$ by an element $s$ in 
$\Delta^3_{n-2}(\Z)\subset\Z^4$, i.e.~$S=s+\Delta^3_2$.
\end{definition}

We shall identify the edges of an ordered simplex with $\dot\Delta^3_2(\Z)$, 
e.g.~the edges $\varepsilon_{01}$ and $\varepsilon_{12}$ correspond to $(1100)$ 
and $(0110)$.

\begin{definition}
\label{shapeassignment}
A shape assignment on $\Delta^3_n$ is an assignment 
\begin{equation}
z\colon\Delta^3_{n-2}(\Z)\times \dot{\Delta}^3_2(\Z)\to\C\setminus\{0,1\}, 
\qquad (s,e)\mapsto z^e_s
\end{equation}
satisfying the \emph{shape parameter relations}

\begin{equation}
\label{eqn:shapeparamrel}
z^{\varepsilon_{01}}_s=z^{\varepsilon_{23}}_s=\frac{1}{1-z^{\varepsilon_{02}}_s},
\quad z^{\varepsilon_{12}}_s=z^{\varepsilon_{03}}_s=\frac{1}{1-z^{\varepsilon_{01}}_s},
\quad z^{\varepsilon_{02}}_s=z^{\varepsilon_{13}}_s=\frac{1}{1-z^{\varepsilon_{12}}_s}
\end{equation}

\end{definition}
One may think of a shape assignment as an assignment of shape parameters to 
the edges of each subsimplex. The ad hoc indexing of the shape parameters by 
$z$, $z^\prime$ and $z^{\prime\prime}$ is replaced by an indexing scheme, in 
which a shape parameter is indexed according to the edge to which it is 
assigned.

\begin{definition}
\label{def:IntegralPoint}
An \emph{integral point} of $\mathcal{T}$ is an equivalence class of points 
in $\Delta_n^3(\Z)$ identified by the face pairings of $\T$. We view an 
integral point as a set of pairs $(t,\Delta)$ with $t\in\Delta^3_n(\Z)$ and 
$\Delta\in\mathcal T$. An integral point is either a \emph{vertex point}, an 
\emph{edge point}, a \emph{face point}, or an \emph{interior point}.
\end{definition}

\begin{definition}
\label{def:GluingEquationsOri}
A shape assignment on $\mathcal{T}$ is a shape assignment $z^e_{s,\Delta}$ on 
each simplex $\Delta\in\mathcal{T}$ such that for each non-vertex integral 
point $p$, the \emph{generalized gluing equation}
\begin{equation}
\label{eq:DefGeneralizedGluingEq}
\prod\limits_{(t,\Delta)\in p}\,\prod\limits_{t = s + e} z^e_{s,\Delta} = 1. 
\end{equation}
is satisfied. 
\end{definition}

The gluing equation for $p$ sets equal to $1$ the product of the shape 
parameters of all edges of subsimplices having $p$ as midpoint, see 
Figures~\ref{fig:EdgeEquation} and \ref{fig:FaceEquation} (taken 
from~\cite{GaroufalidisGoernerZickert}).

\begin{figure}[htb]
\begin{center}
\begin{minipage}[c]{0.44\textwidth}
\scalebox{0.37}{\input{tetGluing1.tex}}
\end{minipage}
\begin{minipage}[c]{0.44\textwidth}
\input{tetFaceGluingEquation.tex}
\end{minipage}
\hfill
\\
\begin{minipage}[t]{0.5\textwidth}
\caption{Edge equation for $n=5$:\\
$z_{1200,0}^{1100}z_{0102,1}^{0101}z_{0120,2}^{0110}=1$.}
\label{fig:EdgeEquation}
\end{minipage}
\hspace{-0.3cm}
\begin{minipage}[t]{0.5\textwidth}
\caption{Face equation for $n=6$:\\ 
$z^{0011}_{2011,0}z^{1001}_{1021,0}z^{1010}_{1012,0}z^{0011}_{0211,1}
z^{0101}_{0121,1}z^{0110}_{0112,1}=1$.}
\label{fig:FaceEquation}
\end{minipage}
\end{center}
\end{figure}

Note that the gluing equation for $p$ can be written in the form
\begin{equation}.
\prod_{(s,\Delta)}(z_{(s,\Delta)})^{A_{p,(s,\Delta)}}\prod_{(s,\Delta)} 
(1-z_{(s,\Delta)})^{B_{p,(s,\Delta)}}=\varepsilon_p
\end{equation}

\begin{theorem}
[{Garoufalidis--Goerner--Zickert~\cite{GaroufalidisGoernerZickert}}]
A shape assignment on $\T$ determines a representation 
$\pi_1(M)\to \PGL(n,\C)$. 
\end{theorem}

\subsection{$X$-coordinates}

The $X$-coordinates are defined on the face points of $\T$, and are used 
in Section~\ref{sec:CuspEquations} to define the cusp equations. They agree 
with the $X$-coordinates of Fock and Goncharov~\cite{FockGoncharov}.

\begin{definition} 
Let $z$ be a shape assignment on $\Delta^3_n$ and let $t\in\Delta^3_n(\Z)$ 
be a face point. The \emph{$X$-coordinate} at $t$ is given by
\begin{equation}
\label{eq:XandProducts}
X_t=-\prod_{s+e=t}z^e_s,
\end{equation}
i.e.~it equals (minus) the product of the shape parameters of the $3$ edges 
of subsimplices having $t$ as a midpoint. 
\end{definition} 

\begin{remark} 
Note that the gluing equation for a face point 
$p=\{(t_1,\Delta_1),(t_2,\Delta_2)\}$ states that $X_{t_1}X_{t_2}=1$.
\end{remark}



\section{Definition of the chain complex}
\label{sec:ChainComplexDefn}

Let $C_0^{\fg}(\T)=C_0(\T)\otimes \Z^{n-1}$. Letting $e_1,\dots, e_n$, denote 
the standard basis vectors of $\Z^{n-1}$, it follows that $C_0^{\fg}(\T)$ is 
generated by symbols $x\otimes e_i$, where $x$ is a $0$-cell of $\T$. It will 
occasionally be convenient to define $e_0=e_n=0$. Let $C_1^{\fg}(\T)$ be the 
free abelian group on the non-vertex integral points of $\T$, and let 
\begin{equation}
J^{\fg}(\T)=\bigoplus_{\Delta\in\T}\bigoplus_{s\in\Delta_{n-2}^3}J_{\Delta_2^3},
\end{equation}
be a direct sum of copies of $J_{\Delta^3_2}$, one for each subsimplex of each 
simplex of $\T$. Note that $J^{\fg}(\T)$ is generated by the set of all edges 
of all subsimplices of the simplices of $\T$. We denote a generator by 
$(s,e)_\Delta$. The form $\Omega$ on $J_{\Delta_2^3}$ induces by orthogonal 
extension a form on $J^{\fg}(\T)$ also denoted by $\Omega$. Since $\Omega$ is 
non-degenerate it induces a natural identification of $J^{\fg}(\T)$ with its 
dual. Similarly, the natural bases of $C_0^{\fg}(\T)$ and $C_1^{\fg}(\T)$ induce 
natural identifications with their respective duals. 

\subsection{Formulas for $\beta$ and $\beta^*$}

Define
\begin{equation}
\label{eq:DefBeta}
\beta\colon C^{\fg}_{1}(\T)\to J^{\fg}(\T),\qquad 
p=\{(t,\Delta)\}\mapsto\sum_{(\Delta,t)\in p}\sum_{e+s=t}(s,e)_\Delta.
\end{equation}
Hence, $\beta$ takes $p$ to the formal sum of all the edges of subsimplices 
whose midpoint is $p$.
By \cite[Lemma~7.3]{GaroufalidisGoernerZickert}, the dual map 
$\b^*\colon  J^{\fg}(\T) \to  C^{\fg}_1(\T)$ is given by
\begin{equation}
\label{eq:BetaStar}
\begin{aligned}
(s,\varepsilon_{01})_\Delta&
\mapsto[(s+\varepsilon_{03},\Delta)]+[(s+\varepsilon_{12},\Delta)]
-[(s+\varepsilon_{02},\Delta)]-[(s+\varepsilon_{13},\Delta)]\\
(s,\varepsilon_{12})_\Delta&
\mapsto[(s+\varepsilon_{02},\Delta)]+[(s+\varepsilon_{13},\Delta)]
-[(s+\varepsilon_{01},\Delta)]-[(s+\varepsilon_{23},\Delta)]\\
(s,\varepsilon_{02})_\Delta&
\mapsto[(s+\varepsilon_{01},\Delta)]+[(s+\varepsilon_{23},\Delta)]
-[(s+\varepsilon_{23},\Delta)]-[(s+\varepsilon_{12},\Delta)].
\end{aligned}
\end{equation}

We refer to an element of the form $\beta^*(s,\varepsilon_{ij})$ as an 
\emph{elementary quad relation}, see 
Figures~\ref{fig:Quad01},~\ref{fig:Quad12} and \ref{fig:Quad02}.

\begin{figure}[htb]
\begin{center}
\begin{minipage}[c]{0.3\textwidth}
\scalebox{0.6}{\input{Quad01.tex}}
\end{minipage}
\begin{minipage}[c]{0.3\textwidth}
\scalebox{0.6}{\input{Quad12.tex}}
\end{minipage}
\begin{minipage}[c]{0.3\textwidth}
\scalebox{0.6}{\input{Quad02.tex}}
\end{minipage}
\\
\hspace{-0.3cm}
\begin{minipage}[t]{0.33\textwidth}
\caption{$\beta^*(s,\varepsilon_{01})$.}\label{fig:Quad01}
\end{minipage}
\hspace{-0.4cm}
\begin{minipage}[t]{0.33\textwidth}
\caption{$\beta^*(s,\varepsilon_{12})$.}\label{fig:Quad12}
\end{minipage}
\hspace{-0.5cm}
\begin{minipage}[t]{0.33\textwidth}
\caption{$\beta^*(s,\varepsilon_{02})$.}\label{fig:Quad02}
\end{minipage}
\end{center}
\end{figure}

\begin{lemma}
[{Garoufalidis--Goerner--Zickert~\cite[Prop.~7.4]{GaroufalidisGoernerZickert}}]
$\b^*\circ\b=0$.\qed
\end{lemma}

\subsection{Formulas for $\alpha$ and $\alpha^*$}

For a $0$-cell $x$ of $\T$ and a simplex $\Delta$, let 
$I_\Delta(x)\subset\{0,1,2,3\}$ be the set of vertices of $\Delta$ that are 
identified with $x$. Also, for $t\in\Delta_n^3(\Z)$ and $k\in\{1,\dots,n-1\}$, 
let 
\begin{equation}
c_{t,\Delta,k}=\left\vert\{i\in I_{\Delta}(x)\mid t_i=k\}\right\vert.
\end{equation}
Note that if $(t,\Delta)$ and $(t^\prime,\Delta^\prime)$ define the same 
integral point, then $c_{t,\Delta,k}=c_{t^\prime,\Delta^\prime,k}$. Define
\begin{equation}
\alpha\colon C_0^{\fg}(\T)\to C_1^{\fg}(\T),\qquad 
x\otimes e_k\mapsto \sum_{p}c_{t,\Delta,k}p.
\end{equation}
Also, define
\begin{equation}
\label{eq:alpha*}
\a^*\colon C^{\fg}_1(\CT) \to C^{\fg}_0(\CT), \qquad
[(t,\Delta)]\mapsto\sum_{i=0}^3 x_i\otimes e_{t_i},
\end{equation}
where $x_i$ is the $0$-cell of $\T$ defined by the $i$th vertex of $\Delta$.
It is elementary to check that $\alpha^*$ is well defined, and that it is 
the dual of $\alpha$.


\begin{lemma}
We have $\a^*\circ\b^*=0$.
\end{lemma}

\begin{proof} 
Let $s\in\Delta^3_{n-2}(\Z)$ be a subsimplex. We have
\begin{equation*}
\begin{aligned}
\a^* \circ \b^*(s,\varepsilon_{01})_\Delta &=
\a^*([(s+\varepsilon_{03},\Delta)])
+\a^*([(s+\varepsilon_{12},\Delta)])
\\
&
-\a^*([(s+\varepsilon_{02},\Delta)])
-\a^*([(s+\varepsilon_{13},\Delta)]) 
\\ &
=x_0\otimes e_{s_0+1}+x_1\otimes e_{s_1}+x_2\otimes e_{s_2}+x_3\otimes e_{s_3+1} 
\\ &
+ x_0\otimes e_{s_1}+x_1\otimes e_{s_1+1}+x_2\otimes e_{s_2+1}+x_3\otimes e_{s_3} 
\\ & 
-x_0\otimes e_{s_0+1}-x_1\otimes e_{s_1}-x_2\otimes e_{s_2+1}-x_3\otimes e_{s_3} 
\\ & 
- x_0\otimes e_{s_0}-x_1\otimes e_{s_1+1}-x_2\otimes e_{s_2}-x_3\otimes e_{s_3+1} 
=0
\end{aligned}
\end{equation*}
Likewise, $\a^* \circ \b^*(s,\varepsilon_{12})_\Delta
=\a^* \circ \b^*(s,\varepsilon_{02})_\Delta=0$.
\end{proof}
By duality, $\beta\circ\alpha$ is also $0$, so we have a chain complex 
$\J^{\fg}(\T)$:
\begin{equation}
\xymatrix{0 \ar[r] & C^{\fg}_0(\CT) \ar[r]^\alpha & C^{\fg}_1(\CT)\ar[r]^\beta
& J^{\fg}(\T)\ar[r]^{\beta^*} & C^{\fg}_1(\CT) \ar[r]^{\alpha^*} & 
C^{\fg}_0(\CT) \ar[r] &0}
\end{equation}
Note that when $n=2$, $\J^{\fg}$ equals $\J$.

\begin{convention}
When there can be no confusion, we shall sometimes suppress the simplex 
$\Delta$ from the notation. For example, we sometimes write $(s,e)$ instead 
of $(s,e)_\Delta$, and if $t$ is an integral point of a simplex $\Delta$ of 
$\T$, we denote the corresponding integral point of $\T$ by $[t]$ or 
sometimes just $t$ instead of $[(t,\Delta)]$.
\end{convention}


\section{Characterization of $\Im(\beta^*)$}

We develop some relations in $C_1^{\fg}(\T)/\Im(\beta^*)$ that are needed for 
computing $H_2(\J^{\fg})$. These relations may be of independent interest.

\subsection{Quad relations}

\begin{definition}
\label{def:QuadRelation}
A \emph{quadrilateral} (\emph{quad} for short) in $\Delta_n^3$ is the convex 
hull of $4$ points
\begin{equation}
p_0=a+(k,0,0,l),\quad p_1=a+(k,0,l,0),\quad p_2=a+(0,k,l,0),\quad 
p_3=a+(0,k,0,l),
\end{equation}
or the image of such under a permutation in $S_4$. Here $k,l$ are positive 
integers with $k+l\leq n$ and $a\in\Delta_{n-k-l}(\Z)$.
A quad determines a \emph{quad relation} in $C_1^{\fg}(\T)$ given by the 
alternating sum $p_0-p_1+p_2-p_3$ of its corners.
\end{definition}

Figure~\ref{fig:QuadRelations} shows $3$ quad relations for $n=4$.

\begin{lemma}
A quad relation is in the image of $\beta^*$, and is thus zero in 
$H_2(\J^{\fg})$.
\end{lemma}

\begin{proof}
Algebraically, we have
\begin{equation}
p_0-p_1+p_2-p_3=\sum_{1<i\leq k,l<j\leq l} \beta^*\big(a+(k-i,i-1,j-1,l-j),
\varepsilon_{01}\big).
\end{equation}
For a geometric proof, note that any quad relation is a sum of the 
elementary quad relations in Figures~\ref{fig:Quad01},~\ref{fig:Quad12} 
and \ref{fig:Quad02}.
\end{proof}

\begin{figure}[htb]
\begin{center}
\begin{minipage}[c]{0.3\textwidth}
\scalebox{0.7}{\input{QuadRelations.tex}}
\end{minipage}
\hspace{0.5cm}
\begin{minipage}[c]{0.3\textwidth}
\scalebox{0.65}{\input{Hexagon.tex}}
\end{minipage}
\hspace{0.5cm}
\begin{minipage}[c]{0.3\textwidth}
\scalebox{0.65}{\input{LongHexagon.tex}}
\end{minipage}
\\
\hspace{-1cm}
\begin{minipage}[t]{0.33\textwidth}
\caption{Quad relations.}\label{fig:QuadRelations}
\end{minipage}
\begin{minipage}[t]{0.33\textwidth}
\caption{Hexagon relation.}\label{fig:HexagonRelation}
\end{minipage}
\hspace{0.2cm}
\begin{minipage}[t]{0.33\textwidth}
\caption{Long hexagon relation.}\label{fig:LongHexagonRelation}
\end{minipage}
\end{center}
\end{figure}

Recall that we have divided integral points into edge points, face points 
and interior points. We shall need a finer division.
\begin{definition}
The \emph{type} of a point $t\in\Delta_n(\Z)$ is the orbit of $t$ under the 
$S_4$ action. 
\end{definition}
Note that the type is preserved under face pairings, so it makes sense to 
define the type of an integral point $p=\{(t,\Delta)\}$ to be the type of 
any representative.

\begin{proposition}
\label{prop:SameTypeZero}
Let $p$ and $q$ be integral points of the same type. Then 
\begin{equation}
p-q\in \Im(\beta^*)+E,
\end{equation}
where $E$ is the subgroup of $C_1^{\fg}(\T)$ generated by edge points.
\end{proposition}

\begin{proof}
We may assume that $p$ and $q$ lie in the same simplex. The quad relation 
(together with similar relations obtained by permutations)
\begin{equation}
(a_1,a_0,a_2,a_3)+(a_0,a_1,a_2+a_3,0)-(a_1,a_0,a_2+a_3,0)-(a_0,a_1,a_2,a_3)
\end{equation} 
shows that the difference between two interior points is equal modulo 
$\Im(\beta^*)$ to the difference between two face points.
Similarly, the relation
\begin{equation}
(a,b,0,c)-(a,b,c,0)+(0,a+b,c,0)-(0,a+b,0,c)
\end{equation}
shows that the difference between two face points in distinct faces is in 
$\Im(\beta^*)+E$.
Finally, the two quad relations
\begin{equation}
\begin{gathered}
(0,a,b,c)=(a,0,b,c)+(0,a,0,b+c)-(a,0,0,c+b),\\
(0,a,c,b)=(a,0,c,b)+(0,a,b+c,0)-(a,0,b+c,0)
\end{gathered}
\end{equation}
in $C_1^{\fg}(\T)/\Im(\beta^*)$ imply that the difference between two face 
points in the same face is also in $\Im(\beta^*)+E$. This concludes the proof.
\end{proof}

\subsection{Hexagon relations}

Besides the quad relations, we shall need further relations that lie entirely 
in a face.

\begin{lemma}
\label{lemma:HexagonRelation} 
For any face point $t$, the element $\beta^*\big(\sum_{s+e=t}(s,e)\big)$ is 
an alternating sum of the corners of a hexagon with center at $t$ (see 
Figure~\ref{fig:HexagonRelation}). 
\end{lemma}

\begin{proof}
By rotational symmetry, we may assume that $t=(t_0,t_1,t_2,0)$. We thus have
\begin{equation}
\label{eq:wFormula}
\beta^*\big(\sum_{s+e=t}(s,e)\big)=
\beta^*(t-\varepsilon_{01},\varepsilon_{01})
+\beta^*(t-\varepsilon_{12},\varepsilon_{12})
+\beta^*(t-\varepsilon_{02},\varepsilon_{02}).
\end{equation}
Using the formula~\eqref{eq:BetaStar} for $\beta^*$, \eqref{eq:wFormula} 
easily simplifies to 
\begin{multline}
\beta^*\big(\sum_{s+e=t}(s,e)\big)
=-[t+(-1,1,0,0)]+[t+(-1,0,1,0)]-[t+(0,-1,1,0)]\\
+[t+(1,-1,0,0)]-[t+(1,0,-1)]+[t+(0,1,-1,0)].
\end{multline}
This corresponds to the configuration in Figure~\ref{fig:HexagonRelation}.
\end{proof}

\begin{definition}
An element as in Lemma~\ref{lemma:HexagonRelation} is called a 
\emph{hexagon relation}. By taking sums of hexagon relation, we obtain 
relations as shown in Figure~\ref{fig:LongHexagonRelation}. We refer to 
these as \emph{long hexagon relations} (a hexagon relation is also regarded 
as a long hexagon relation). 
\end{definition}


\section{The outer homology groups}
\label{sec:OuterHomology}

We focus on here on the computation of $H_1(\J^{\fg})$ and $H_2(\J^{\fg})$; 
the computation of $H_5(\J^{\fg})$ and $H_4(\J^{\fg})$ will follow by a duality 
argument (see Section~\ref{sec:UniversalCoeff}).


\subsection{Computation of $H_1(\J^{\fg})$}

\begin{proposition}
\label{prop:H1}
$H_1(\J^{\fg})=\Z/n\Z$.
\end{proposition}

\begin{proof}
Consider the map 
$$
\ep\colon C^{\fg}_0(\CT) \to \BZ/n, \qquad x\otimes e_k\mapsto k.
$$
One easily checks that $\Im(\a^*) \subset
\Ker(\e)$. To prove the other inclusion, let $[\sigma]\in H_1(\J)$, and 
$\sigma=\sum_{i=1}^N\varepsilon_i x_i\otimes e_{k_i}$ a representative with 
$N$ minimal and $\varepsilon_i=\pm 1$ . We wish to prove that $N=0$, so 
suppose $N>0$. We start by showing that modulo $\Im(\a^*)$
\begin{equation}
\label{eq:Auxiliary}
x\otimes e_k +y\otimes e_{n-k}=0,\qquad x\otimes e_k-y\otimes e_k=0.
\end{equation}
Pick an edge path of odd length between $x$ and $y$ with vertices 
$x_0=x,x_1,\dots,x_{2k-1}=y$. For $z,w$ vertices joined by an edge $e$, let 
$(z,w;k)$ be the edge point on $e$ at distance $k$ from $w$. Then 
$\alpha^*(z,w;k)=z\otimes e_k+w\otimes e_{n-k}$. We thus have
\begin{equation}
x\otimes e_k+y\otimes e_{n-k}=\alpha^*\big((x,x_1;k)-(x_1,x_2;n-k)
+\dots+(x_{2k-2},y;k)\big).
\end{equation}
This proves the first equation in~\eqref{eq:Auxiliary}. The second follows 
similarly by considering an edge path of even length.
We may thus assume that $N\geq 3$. Without loss of generality 
$\varepsilon_0=1$, so that
\begin{equation}
\sigma=x_0\otimes e_{k_0}+\varepsilon_1 x_1\otimes e_{k_1}
+\sum_{i=2}^n\varepsilon_i x_i\otimes e_{k_i}.
\end{equation}

Using~\eqref{eq:Auxiliary}, we may assume that $k_0+k_1\leq n$ if 
$\varepsilon_1=1$ and $k_1\geq k_0$ if $\varepsilon_1=-1$.
Fix three $0$-cells $x,y,z$ lying on a face, and let $p$ be the unique 
integral point satisfying
\begin{equation}
\alpha^*(p)=\begin{cases} 
x\otimes e_{k_0}+y\otimes e_{k_1}+z\otimes e_{n-k_0-k_1}&
\text{ if }\varepsilon_1=1\\   
x\otimes e_{k_0}+y\otimes e_{n-k_1}+z\otimes e_{k_1-k_0}&
\text{ if }\varepsilon_1=-1.
\end{cases}
\end{equation}
Subtracting $\alpha^*(p)$ from $\sigma$ and using \eqref{eq:Auxiliary}, we 
can thus construct a representative of $[\sigma]$ with fewer than $N$ terms, 
contradicting minimality of $N$. Hence, $\sigma=0$.
\end{proof}


\subsection{Computation of $H_2(\J^{\fg})$}

In this section we prove that $H_2(\J^{\fg})=H_1(\widehat M;\Z/n\Z)$. The 
fact that $H_2(\J^{\fg})$ is torsion is crucial, and is used in the proof of 
Proposition~\ref{prop:SameRank}. We see no way of proving that $H_2(\J^{\fg})$ 
is torsion without computing it explicitly.

We assume for convenience that the triangulation is ordered, i.e.~that the 
face pairings are order preserving. The general case differs only in notation.
Let $\varepsilon_{ij}^{\ori}$ denote the \emph{oriented} edge (from $i$ to $j$) 
between $i$ and $j$.

\subsubsection{Definition of a map 
$\nu\colon H_2(\J^{\fg})\to H_1(\widehat M;\Z/n\Z)$}

Consider the map 
\begin{equation}
\nu\colon \Z[\dot \Delta^3_n(\Z)]\to C_1(\Delta^3;\Z/n\Z),\qquad 
(t_0,t_1,t_2,t_3)\mapsto 
t_1\varepsilon^{\ori}_{01}+t_2\varepsilon^{\ori}_{02}+t_3\varepsilon^{\ori}_{03}.  
\end{equation}
Modulo boundaries in $C_1(\Delta^3;\Z/n\Z)$, we have
\begin{equation}
t_1\varepsilon^{\ori}_{01}+t_2\varepsilon^{\ori}_{02}+t_3\varepsilon^{\ori}_{03}
=t_0\varepsilon^{\ori}_{10}+t_2\varepsilon^{\ori}_{12}+t_3\varepsilon^{\ori}_{13}
=t_0\varepsilon^{\ori}_{20}+t_1\varepsilon^{\ori}_{21}+t_3\varepsilon^{\ori}_{23}
=t_0\varepsilon^{\ori}_{30}+t_1\varepsilon^{\ori}_{31}+t_2\varepsilon^{\ori}_{32}.
\end{equation}

\begin{lemma} 
The map
\begin{equation}
\label{eq:NuDefinition}
\nu\colon C_1^{\fg}(\T)\to C_1(\widehat M;\Z/n\Z)\big/\{\text{boundaries}\}
\end{equation}
induced by~\eqref{eq:NuDefinition} takes cycles to cycles and boundaries to $0$.
\end{lemma}

\begin{proof}
To see that cycles map to cycles consider the diagram
\begin{equation}
\label{eq:CyclesToCyclesNu}
\cxymatrix{{C_1^{\fg}(\T)\ar[r]^-{\nu}\ar[d]^-{\alpha^*}&
C_1(\widehat M;\Z/n\Z)\big/\{\text{boundaries}\}\ar[d]^-{\partial}\\
C_0^{\fg}(\T)\ar[r]^-{\nu_0}&C_0(\widehat M;\Z/n\Z),}}
\end{equation}
where $\nu_0$ is the map given by
\begin{equation}
\nu_0\colon C_0^{\fg}(\T)\to C_0(\widehat M;\Z/n\Z),\qquad 
x\otimes e_i\mapsto ix.
\end{equation}
We must prove that~\eqref{eq:CyclesToCyclesNu} is commutative. This follows 
from
\begin{multline}
\partial(\nu(t))=\partial(t_1\varepsilon^{\ori}_{01}+t_2\varepsilon^{\ori}_{02}
+t_3\varepsilon^{\ori}_{03})=\\
t_1(x_1-x_0)+t_2(x_2-x_0)+t_3(x_3-x_0)=t_0x_0+t_1x_1+t_2x_2+t_3x_3=\alpha^*(t).
\end{multline}

We must check that image of $J^{\fg}(\T)$ maps to $0$. By rotational symmetry, 
it is enough to prove that $\nu$ takes $\beta^*(s,\varepsilon_{01})$ to $0$.
Using~\eqref{eq:BetaStar} we have
\begin{multline}
\nu\big(\beta^*(s,\varepsilon_{01})\big)
=\big(s_1\varepsilon^{\ori}_{01}+s_2\varepsilon^{\ori}_{02}
+(s_3+1)\varepsilon^{\ori}_{03}\big)+\big((s_1+1)\varepsilon^{\ori}_{01}
+(s_2+1)\varepsilon^{\ori}_{02}+s_3\varepsilon^{\ori}_{03}\big)\\
-\big(s_1\varepsilon^{\ori}_{01}+(s_2+1)\varepsilon^{\ori}_{02}
+s_3\varepsilon^{\ori}_{03}\big)-\big((s_1+1)\varepsilon^{\ori}_{01}
+s_2\varepsilon^{\ori}_{02}+(s_3+1)\varepsilon^{\ori}_{03}\big)=0.
\end{multline}
This concludes the proof.
\end{proof}
Hence, $\nu$ induces a map 
\begin{equation}
\nu\colon H_2(\J^{\fg})\to H_1(\widehat M;\Z/n\Z).
\end{equation}

\subsubsection{Construction of a map 
$\mu\colon H_1(\widehat M;\Z/n\Z)\to H_2(\J^{\fg})$}

We prove that $\nu$ is an isomorphism by constructing an explicit inverse. 
Let $k\in\{1,2,\dots,n-1\}$. 

\begin{definition} 
Let $e$ be an oriented edge of $\T$. If $f$ is a face 
containing $e$, the path consisting of the two other edges in $f$ is called 
a \emph{tooth} of $e$.
\end{definition}

Given a tooth $T_e$ of an edge $e$, let $\mu_k(e)_{T_e}\in C_1^{\fg}$ be the 
element shown in Figure~\ref{fig:Tooth}. 

\begin{lemma}
\label{lemma:FlipTeeth}
For any two teeth $T_e$ and $T_e^{\prime}$ of $e$, we have
\begin{equation}
\mu_k(e)_{T_e}=\mu_k(e)_{T_e^{\prime}}\in C_1^{\fg}(\T)\big/\Im(\beta^*).
\end{equation}
\end{lemma}

\begin{proof}
Since any two teeth of $e$ are connected through a sequence of flips past 
a simplex in the link of $e$, it is enough to prove the result when $T_e$ and 
$T^\prime_e$ are teeth in a single simplex. Hence, we must prove that a 
configuration as in Figure~\ref{fig:TwoTeeth} represents zero 
in~$C_1^{\fg}(\T)\big/\Im(\beta^*)$. This is a consequence of the quad relation.
\end{proof}

\begin{figure}[htb]
\begin{center}
\begin{minipage}[b]{0.47\textwidth}
\begin{center}
\scalebox{0.6}{\input{MukTooth.tex}}
\end{center}
\end{minipage}
\hfill
\begin{minipage}[b]{0.47\textwidth}
\begin{center}
\scalebox{0.57}{\input{IndependentOfTooth.tex}}
\end{center}
\end{minipage}\\
\begin{minipage}[t]{0.47\textwidth}
\begin{center}
\caption{A tooth $T_e$ of $e$ and $\mu_k(e)_{T_e}$.}\label{fig:Tooth}
\end{center}
\end{minipage}
\hfill
\begin{minipage}[t]{0.47\textwidth}
\begin{center}
\caption{$\mu_k(e)_{T_e}-\mu_k(e)_{T_e^\prime}$ is a quad relation.}
\label{fig:TwoTeeth}
\end{center}
\end{minipage}
\end{center}
\vspace{-2mm}
\end{figure}

It follows that we have a map
\begin{equation}
\mu_k\colon C_1(\widehat M)\to C_1^{\fg}(\T),\qquad e\mapsto \mu_k(e)_{T_e}.
\end{equation}
We shall also consider the map $\overline\mu_k\colon C_1(\widehat M)\to 
C_1^{\fg}(\T)$ taking an oriented edge $e$ of $\T$ to the integral point on 
$e$ at distance $k$ from the initial point of $e$. Note that if $f_1$ and 
$f_2$ are the first and second edge of some tooth of $e$, 
$\mu_k(e)=\overline\mu_k(f_1)-\overline\mu_{n-k}(f_2)$. This is immediate 
from the definition of $\mu_k$ and $\overline\mu_k$.

\begin{lemma}
\label{lemma:ConsecutiveEdges} 
If $e_1$ and $e_2$ are two succesive oriented edges,
\begin{equation}\mu_k(e_1+e_2)=\overline\mu_k(e_1)
-\overline\mu_{n-k}(e_2)\in C_1^{\fg}(\T)\big/\Im(\beta^*).
\end{equation}
\end{lemma}

\begin{proof}
We must show that a configuration as in Figure~\ref{fig:Butterfly1} 
represents $0$ in $C_1^{\fg}(\T)\big/\Im(\beta^*)$. By flipping the teeth 
of $e_1$ and $e_2$ (which by Lemma~\ref{lemma:FlipTeeth} does not change the 
element in $C_1^{\fg}(\T)\big/\Im(\beta^*)$), we can tranform the 
configuration into a configuration as in Figure~\ref{fig:FlippedButterfly1} 
where the two teeth meet at a common edge $e$. This configuration also 
represents $\mu_k(e)_{T_e}-\mu_k(e)_{T_e^{\prime}}$ for two teeth $T_e$ and 
$T_e^{\prime}$ of $e$, so is zero by Lemma~\ref{lemma:FlipTeeth}.
\end{proof}

\begin{figure}[htb]
\begin{center}
\begin{minipage}[b]{0.47\textwidth}
\begin{center}
\scalebox{0.5}{\input{Butterfly1.tex}}
\end{center}
\end{minipage}
\hfill
\begin{minipage}[b]{0.47\textwidth}
\begin{center}
\scalebox{0.57}{\input{FlippedButterfly1.tex}}
\end{center}
\end{minipage}\\
\begin{minipage}[t]{0.47\textwidth}
\begin{center}
\caption{Configuration representing 
$\mu_k(e_1+e_2)-\overline\mu_k(e_1)+\overline\mu_{n-k}(e_2)$.}
\label{fig:Butterfly1}
\end{center}
\end{minipage}
\hfill
\begin{minipage}[t]{0.47\textwidth}
\begin{center}
\caption{Configuration representing $\mu_k(e)_{T_e}-\mu_k(e)_{T_e^\prime}=0$.}
\label{fig:FlippedButterfly1}
\end{center}
\end{minipage}
\end{center}
\end{figure}

\begin{corollary}
$\mu_k$ induces a map $\mu_k\colon H_1(\widehat M)\to H_2(\J^{g})$.
\end{corollary}

\begin{proof}
The fact that $\mu_k$ takes cycles to cycles is immediate from the 
definition of $\alpha^*$. Let $e_1+e_2+e_3$ be an oriented path representing 
the boundary of a face in $\T$. We have 
\begin{multline}
\mu_k(e_1+e_2+e_3)=\mu_k(e_1+e_2)+\mu_k(e_3)=\\
\overline\mu_k(e_1)-\overline\mu_{n-k}(e_2)+\mu_k(e_3)=-\mu_k(e_3)+\mu_k(e_3)=0,
\end{multline}
where the third equality follows from the fact that $e_1+e_2$ is a tooth of 
$e_3$. This proves the result.
\end{proof}

\begin{lemma}
\label{lemma:FirstProperty}
We have $\mu_k=-\mu_{n-k}\colon H_1(\widehat M)\to H_2(\J^{\fg})$.
\end{lemma}

\begin{proof}
Let $\alpha\in H_1(\widehat M)$. Since $H_1(\widehat M)$ is generated by 
edge cycles, we may assume that $\alpha $ is represented by an edge cycle
$e_1+e_2+\dots + e_{2l}$, which we may assume to have even length. We thus 
have (indices modulo $2l$)
\begin{equation}
\label{eq:ShiftByOne}
\mu_k(\alpha)=\sum_{i=1}^l\big(\overline\mu_k(e_{2i-1})
-\overline\mu_{n-k}(e_{2i})\big)
=\sum_{i=1}^l\big(-\overline\mu_{n-k}(e_{2i})+\overline\mu_{k}(e_{2i+1})\big)
=-\mu_{n-k}(\alpha),
\end{equation}
where the second and fourth equality follow from 
Lemma~\ref{lemma:ConsecutiveEdges} and the third equality follows from 
shifting indices by $1$.
\end{proof}

\begin{lemma}
\label{lemma:SecondProperty}
For each $k$, we have $\mu_k=k\mu_1\colon H_1(\widehat M)\to H_2(\J^{\fg})$.
\end{lemma}

\begin{proof}
Let $\alpha=e_1+e_2+\dots + e_{2l}$ as in the proof of 
Lemma~\ref{lemma:ConsecutiveEdges}. We can represent $k\mu_1(e_i)-\mu_k(e_i)$ 
as in Figure~\ref{fig:MukEqkMu1}. By applying the long hexagon relations 
($k-d$ relations at distance $d$ from $e_i$), the configuration is 
equivalent to that of Figure~\ref{fig:MukEqkMu1Hexagon}. Now consider two 
consecutive edges $e_i$ and $e_i+1$ as in Figure~\ref{fig:Butterfly2}.
By flipping teeth (which doesn't change the homology class), we may 
transform the configuration into that of Figure~\ref{fig:FlippedButterfly2}, 
and by further flipping, we may assume that the configuration lies in a 
single simplex. It is now evident, that the points near the common edge 
$e$ represents a sum of $k-1$ quad relations. Hence, all the points near 
$e$ vanish. By flipping the teeth back, we end up with a configuration as 
in Figure~\ref{fig:Butterfly2}, but with no points in the middle. Since 
$\alpha$ is a cycle, it follows that everything sums to zero.  
\end{proof}

\begin{figure}[htb]
\begin{center}
\begin{minipage}[b]{0.47\textwidth}
\begin{center}
\scalebox{0.6}{\input{MukEqkMu1.tex}}
\end{center}
\end{minipage}
\hfill
\begin{minipage}[b]{0.47\textwidth}
\begin{center}
\scalebox{0.57}{\input{MukEqkMu1Hexagon.tex}}
\end{center}
\end{minipage}\\
\begin{minipage}[t]{0.47\textwidth}
\begin{center}
\caption{$k\mu_1(e_i)-\mu_k(e_i)$.}\label{fig:MukEqkMu1}
\end{center}
\end{minipage}
\hfill
\begin{minipage}[t]{0.47\textwidth}
\begin{center}
\caption{$k\mu_1(e_i)-\mu_k(e_i)$ after adding long hexagon relations.}
\label{fig:MukEqkMu1Hexagon}
\end{center}
\end{minipage}
\end{center}
\end{figure}

\begin{figure}[htb]
\begin{center}
\begin{minipage}[b]{0.47\textwidth}
\begin{center}
\scalebox{0.52}{\input{Butterfly2.tex}}
\end{center}
\end{minipage}
\hfill
\begin{minipage}[b]{0.47\textwidth}
\begin{center}
\scalebox{0.57}{\input{FlippedButterfly2.tex}}
\end{center}
\end{minipage}\\
\begin{minipage}[t]{0.47\textwidth}
\begin{center}
\caption{$k\mu_1(e_i)-\mu_k(e_i)$.}\label{fig:Butterfly2}
\end{center}
\end{minipage}
\hfill
\begin{minipage}[t]{0.47\textwidth}
\begin{center}
\caption{$k\mu_1(e_i)-\mu_k(e_i)$ after adding long hexagon relations.}
\label{fig:FlippedButterfly2}
\end{center}
\end{minipage}
\end{center}
\end{figure}

By the above lemmas
we have a map
\begin{equation}
\mu\colon H_1(\widehat M;\Z/n\Z)\to H_2(\J^{\fg}),\qquad 
e\otimes k\mapsto \mu_k(e).
\end{equation}

\subsubsection{$\mu$ is the inverse of $\nu$.}

\begin{lemma}
The composition $\nu\circ\mu$ is the identity on $H_1(\widehat M;\Z/n\Z)$.
\end{lemma}

\begin{proof}
First observe that for each $1$-cell $e$ of $\T$, we have 
$\nu\circ\overline\mu_k(e)=ke$.
Consider a representative $\alpha=e_1+\dots e_{2l}\in C_1(\widehat M;\Z)$ of 
a homology class in $H_1(\widehat M)$. 
As in~\eqref{eq:ShiftByOne}, we have
\begin{multline}
\label{eq:NuMu}
\nu\circ\mu_k(\alpha)=\nu\Big(\sum_{i=1}^l
\big(\overline\mu_k(e_{2i-1})-\overline\mu_{n-k}(e_{2i})\big)\Big)=\\
\sum_{i=1}^l e_{2i-1}\otimes k-e_{2_i}\otimes(n-k)
=\alpha\otimes k\in C_1(\widehat M;\Z/n\Z)\big/\{\text{boundaries}\}.
\end{multline}
This proves the result.
\end{proof}

We now show that $\mu\circ\nu$ is the identity on $H_2(\J^{\fg})$. 
The idea is that every homology class in $H_2(\J^{g})$ can be 
represented by edge points.
Consider the set
\begin{equation}
T=\big\{(a,b,c,0)\in\dot\Delta^3_n(\Z)\bigm\vert a\geq b\geq c\geq 0\big\}.
\end{equation}
By Proposition~\ref{prop:SameTypeZero}, we can represent each homology 
class by an element $\tau+e$, where $e$ consists entirely of edge points, 
and $\tau$ consists of terms $(a,b,c,d)$ with $a\geq b\geq c\geq d\geq 0$ 
lying in a single simplex. Since an interior point is a sum of two face 
points minus an edge point, we may assume that $d=0$, i.e.~that all terms 
of $l$ are in $T$. Note that by adding and subtracting edge points in $T$ 
to $\tau$, one may further assume that $\alpha^*(\tau)=0$. Hence, we shall 
study elements in $\Ker(\alpha^*)$ of the form
\begin{equation}
\tau=\sum_{t\in T} k_tt, \qquad k_t\in\Z.
\end{equation} 
We say that a term $t\in T$ is in $\tau$ if $k_t\neq 0$. For 
$j=1,\dots, n-1$, consider the map 
\begin{equation}
\pi_j\colon C_0^{\fg}(\T)\to \Z,\qquad x\otimes e_i\mapsto \delta_{ij},
\end{equation}
where $\delta_{ij}$ is the Kronecker $\delta$.

\begin{lemma}
\label{lemma:AllRelationsGood}
Let $\tau=\sum_{t\in T}k_t t\in \Ker(\alpha^*)$. For any term $t$ in $\tau$, 
$t_1<t_0^{\max}$.
\end{lemma}

\begin{proof}
We must show that $u=(t_0^{\max},t_0^{\max},n-2t_0^{\max})$ can't be a term in 
$\tau$.  If $2t_0^{\max}>n$, $u\notin T$, so assume that $2t_0^{\max}\leq n$. 
Since $u$ is the unique element in $T$ with $t_2=n-2t_0^{\max}$, we have
$\pi_{n-2t_0^{\max}}\circ\alpha^*(\tau)=k_u$. Since $\alpha^*(\tau)=0$, it 
follows that $k_u=0$. Hence, $u$ is not a term in $\tau$.
\end{proof}

We can write $\tau=\tau^{\max}+\tau^{\rest}$, where $\tau^{\max}$ involves all 
the terms with $t_0=t_0^{\max}$, and $\tau^{\rest}$ involves all the rest.

\begin{lemma}
\label{lemma:AddLongHexagon} 
The maximal part $\tau^{\max}$ of $\tau$ is a linear combination of terms of 
the form 
\begin{equation}
C_{t_1,t_1^{\prime}}=(t_0^{\max},t_1,t_2,0)-(t_0^{\max},t_1^{\prime},t_2^{\prime},0),
\qquad  t_1>t_1^{\prime}.
\end{equation}
\end{lemma}

\begin{proof}
It is enough to prove that $\sum_{t\vert t_0=t_0^{\max}}k_t=0$. Since 
$\alpha^*(\tau)=0$, this follows from
\begin{equation}
0=\pi_{t_0^{\max}}\circ\alpha^*(\tau)=\sum_{t\vert t_0=t_0^{\max}}k_t,
\end{equation}
where the second equality follows directly from the definition of $\alpha^*$.
\end{proof}

\begin{proposition}
\label{prop:EdgePointRep}
The kernel of $\alpha^*\colon C_1^{\fg}(\T)\to C_0^{\fg}(\T)$ is generated modulo 
$\Im(\beta^*)$ by edge points. In other word, each homology class can be 
represented by edge points.
\end{proposition}

\begin{proof}
Let $x\in H_2(\J^{\fg})$. As explained above, we can represent $x$ by an 
element $\tau+e$, where $e$ consists entirely of edge points, and 
$\tau=\sum_{t\in T} k_t t\in \Ker(\alpha^*)$. We wish to show that $\tau$ is 
a linear combination of long hexagon relations. The idea is to add (and 
subtract) long hexagon relations to $\tau$ until we end up with an element 
$\tau^\prime$ with $t_0^{\max}(\tau^{\prime})>t_0^{\max}$. This process can then 
be repeated, and since $t_0<n$ for all $t\in T$, we will eventually end 
with $0$. More precisely, we start by adding (or subtracting the long 
hexagons) with corners at the two terms involved in $C_{t_1,t_1^{\prime}}$ 
from Lemma~\ref{lemma:AddLongHexagon}. If a long hexagon has a vertex 
outside of $t$, this vertex is replaced by the unique vertex in $T$ of the 
same type. Lemma~\ref{lemma:AllRelationsGood} implies that the element 
$\tau^{\prime}$ thus obtained satisfies $t_0^{\max}(\tau^{\prime})>t_0^{\max}$.
The process is illustrated in Figure~\ref{fig:Illustration}. Note that at 
the second step, the required long hexagon has a vertex outside of $T$.
\end{proof}

\begin{figure}
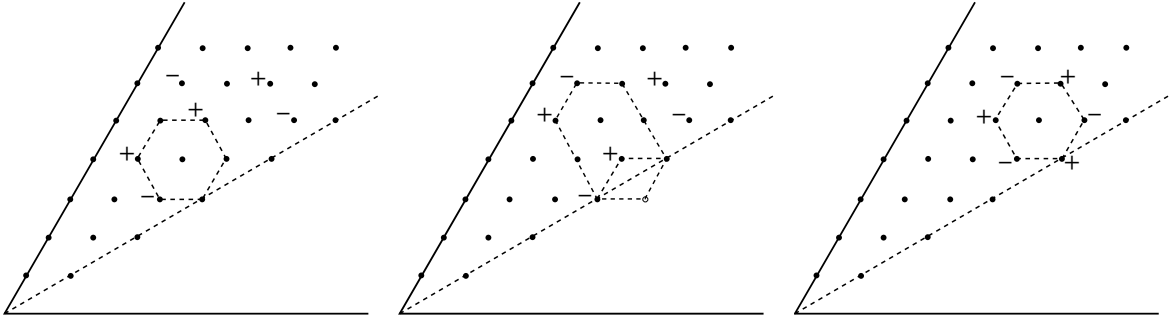

\begin{center}
\subfigure{\scalebox{0.42}{\input{KerAlphaStar1.tex}}}
\subfigure{\scalebox{0.42}{\input{KerAlphaStar2.tex}}}
\subfigure{\scalebox{0.42}{\input{KerAlphaStar3.tex}}}
\end{center}
\caption{Writing $\tau$ as a sum of long hexagon relations.}
\label{fig:Illustration}
\end{figure}

\begin{corollary}
The composition $\mu\circ\nu$ is the identity on $H_2(\J^{\fg})$. 
\end{corollary}

\begin{proof}
By Proposition~\ref{prop:EdgePointRep}, one may represent a class in 
$H_2(\J^{\fg})$ by a linear combination $x$ of edge points. 
Since $\alpha^*(x)=0$, $x$ must be a linear combination of elements of the form 
\begin{equation}
\sigma=\sum_{i=1}{l}\big(\overline\mu_k(e_{2i})-\overline\mu_{n-k}(e_{2i-1})\big) 
\end{equation}
To see this compare with the standard proof that cycles in $C_1(M;\Z)$ are 
generated by edge cycles. We now have
\begin{equation}
\mu\circ\nu(\sigma)=\mu\big((e_1+\dots+e_{2l})\otimes k\big)
=\mu_k(e_1+\dots e_{2l})=\sigma,
\end{equation} 
where the first equality follows from~\eqref{eq:NuMu}, and the third from 
Lemma~\ref{lemma:ConsecutiveEdges}.
\end{proof}

\subsection{Computation of $H_4(\J^{\fg})$ and $H_5(\J^{\fg})$.}
\label{sec:UniversalCoeff}

Since $\J^{\fg}$ is self dual, the universal coefficient theorem implies that
\begin{equation}
\label{eq:SelfDual}
H_k(\J^{\fg})=H_{6-k}((\J^{\fg})^*)\cong \Hom(H_{6-k}(\J^{\fg}),\Z)\oplus
\Ext(H_{6-k-1}(\J^{\fg}),\Z).
\end{equation}
It thus follows that $H_5(\J^{\fg})=0$ and that $H_4(\J^{\fg})=\Z/n\Z$.



\begin{remark}
One can show that the sum $\tau$ of all integral points of $\T$ generates 
$H_4(\J^{\fg})=\Z/n\Z$. If $M$ has a single boundary component, corresponding 
to the $0$-cell $x$ of $\T$, we have
\begin{equation}
n\tau =\alpha\big(\sum_{i=1}^{n-1}i x\otimes e_i\big).
\end{equation} 
We shall not need this, so we leave the proof to the reader.
\end{remark}




\section{The middle homology group}
\label{sec:MiddleHomology}

By~\eqref{eq:SelfDual}, the torsion in $H_3(\J^{\fg})$ equals 
$\Ext(H_1(\widehat M;\Z/n\Z))\cong H^1(\widehat M;\Z/n\Z)$. We now analyze 
the free part. Following Neumann~\cite[Section~4]{NeumannComb}, the idea is 
to construct maps
\begin{equation}
\delta\colon H_1(\partial M;\Z^{n-1})\to H_3(\J^{\fg}),\qquad 
\gamma\colon H_3(\J^{\fg})\to H_1(\partial M;\Z^{n-1}),
\end{equation}
which are adjoint with respect to the intersection form $w$ on 
$H_1(\partial M;\Z^{n-1})$ and the form $\Omega$ on $H_3(\J^{\fg})$. When 
$n=2$, our $\delta$ and $\gamma$ agree with those of \cite{NeumannComb}.

\subsection{Cellular decompositions of the boundary}

The ideal triangulation $\T$ of $M$ induces a decomposition of $M$ into 
truncated simplices such that the cut-off triangles triangulate the boundary 
of $M$. We call this decomposition of $\partial M$ the \emph{standard 
decomposition} and denote it by $\T^\Delta_{\partial M}$. The superscript 
$\Delta$ is to stress that the $2$-cells are triangles. We shall also 
consider another decomposition of $\partial M$, the \emph{polygonal 
decomposition} $\T^{\pentagon}_{\partial M}$, which is obtained from 
$\T^\Delta_{\partial M}$ by replacing the link of each vertex $v$ of 
$\T^\Delta_{\partial M}$ with the polygon whose vertices are the midpoints of 
the edges incident to $v$. The polygonal decomposition thus has a vertex 
for each edge of $\T^\Delta_{\partial M}$, $3$ edges for each face of 
$\T^\Delta_{\partial M}$, and $2$ types of faces; a \emph{triangular face} 
for each face of $\T^\Delta_{\partial M}$, and a \emph{polygonal face} (which 
may or may not be a triangle) for each vertex of $\T^\Delta_{\partial M}$. 

\begin{figure}[htb]
\begin{center}
\begin{minipage}[b]{0.47\textwidth}
\begin{center}
\scalebox{0.65}{\input{StandardDecomposition.tex}}
\end{center}
\end{minipage}
\hfill
\begin{minipage}[b]{0.47\textwidth}
\begin{center}
\scalebox{0.65}{\input{PolygonalDecomposition.tex}}
\end{center}
\end{minipage}\\
\begin{minipage}[t]{0.47\textwidth}
\begin{center}
\caption{The standard decomposition.}
\end{center}
\end{minipage}
\hfill
\begin{minipage}[t]{0.47\textwidth}
\begin{center}
\caption{The polyhedral decomposition.}
\end{center}
\end{minipage}
\end{center}
\end{figure}

We denote the cellular chain complexes corresponding to the two 
decompositions by $C_*(\T^\Delta_{\partial M})$ and 
$C_*(\T^{\pentagon}_{\partial M})$, respectively. Hence, we have canonical 
isomorphisms
\begin{equation}
H_*\big(C_*(\T^{\pentagon}_{\partial M})\big)
=H_*\big(C_*(\T^\Delta_{\partial M})\big)=H_*(\partial M).
\end{equation}

\subsubsection{Labeling and orientation conventions}
\label{sec:LabelingConventions}

We orient $\partial M$ with the counter-clockwise orientation as viewed 
from an ideal point. The edges of $\T^{\pentagon}_{\partial M}$ each lie in a 
unique simplex of $\T$ and we orient them in the unique way that agrees with 
the counter-clockwise orientation for a polygonal face, and the clockwise 
orientation for a triangular face. The triangular faces of 
$\T^{\pentagon}_{\partial M}$ are thus oriented opposite to the orientation 
inherited from $\partial M$.
An edge of $\T^\Delta_{\partial M}$ is only naturally oriented after specifying 
which simplex it belongs to.

We denote the triangular $2$-faces of $\T^{\pentagon}_{\partial M}$ and 
$\T^{\Delta}_{\partial M}$ by $\tau^i_\Delta$ and $T^i_\Delta$, respectively, 
where $i$ is the nearest vertex. The polygonal $2$-face of 
$\T^{\pentagon}_{\partial M}$ whose boundary edges are $e_{\Delta_l}^{i_lj_l}$ is 
denoted by $p^{\{i_l,j_l\}}$. The (oriented) edge of $\T^{\pentagon}_{\partial M}$ 
near vertex $i$ and perpendicular to edge $ij$ of $\Delta$ is denoted by 
$e^{ij}_\Delta$, and the (oriented) edge of $\T^{\Delta}_{\partial M}$ near vertex 
$i$ and parallel to the edge $jk$ of $\Delta$ is denoted by $E^{ijk}_{\Delta}$.  
The vertex of $\T^{\pentagon}_{\partial M}$ near the $i$th vertex of $\Delta$ on 
the face opposite the $j$th vertex is denoted by $v^{ij}_\Delta$, and the 
vertex of $\T^{\Delta}_{\partial M}$ near the $i$th vertex on the edge $ij$ is 
denoted by $V^{ij}_\Delta$. The subscript $\Delta$ will occationally be 
omitted (e.g.~when only one simplex is involved).

\begin{figure}[htb]
\scalebox{1.4}{\input{Labelings.tex}}
\caption{Labeling of vertices, edges and faces of $\T^\Delta_{\partial M}$ and 
$\T^{\pentagon}_{\partial M}$.}
\label{fig:EdgeVertexFace}
\end{figure}

\subsection{The intersection pairing}

Consider the pairing
\begin{equation}
\label{eq:IntersectionPairing}
\iota\colon C_1(\T^{\pentagon}_{\partial M})\times C_1(\T^\Delta_{\partial M})\to \Z,
\end{equation}
given by counting intersections with signs (see Figures~\ref{fig:Positive} 
and \ref{fig:Negative}). We have
\begin{equation}
\iota(e^{ij}_\Delta,E_\Delta^{i^\prime j^\prime k^\prime})=
\begin{cases} 
1&\text{if } i=i^\prime, j=k^\prime\\
-1& \text{if } i=i^\prime, j=j^\prime\\
0& \text{otherwise.}
\end{cases}
\end{equation}

\begin{figure}[htb]
\begin{center}
\begin{minipage}[b]{0.45\textwidth}
\begin{center}
\scalebox{0.6}{\input{PositiveIntersection.tex}}
\end{center}
\end{minipage}
\hfill
\begin{minipage}[b]{0.45\textwidth}
\begin{center}
\scalebox{0.6}{\input{NegativeIntersection.tex}}
\end{center}
\end{minipage}\\
\begin{minipage}[t]{0.45\textwidth}
\begin{center}
\caption{Positive intersection.}
\label{fig:Positive}
\end{center}
\end{minipage}
\hfill
\begin{minipage}[t]{0.45\textwidth}
\begin{center}
\caption{Negative intersection.}
\label{fig:Negative}
\end{center}
\end{minipage}
\end{center}
\end{figure}

\begin{lemma}
The pairing~\eqref{eq:IntersectionPairing} induces the intersection pairing 
on $H_1(\partial M)$.
\end{lemma}

\begin{proof}
Since the pairing counts signed intersection numbers, all we need to prove 
is that pairing cycles with boundaries gives zero. Consider the maps (star 
denotes dual)
\begin{equation}
\Psi\colon C_2(\T^\Delta_{\partial M})\to C_0(\T^{\pentagon})^*,\qquad 
\Phi\colon C_2(\T^{\pentagon}_{\partial M})\to C_0(\T^{\pentagon}_{\partial M})^*,
\end{equation}
where $\Psi$ takes $T^i$ to $-(v^{ij})^*-(v^{ik})^*-(v^{il})^*$ and $\Phi$ 
takes a triangular face to $0$ and a polygonal face to (the dual of) its 
midpoint. One now easily checks that
\begin{equation}
\iota(\partial\tau,E)=\Phi(\tau)(\partial E),\qquad 
\iota(e,\partial T)=\Psi(T)(\partial e)
\end{equation}
from which the result follows.
\end{proof}

Consider the chain complexes
\begin{equation}
C_*(\T^{\pentagon}_{\partial M};\Z^{n-1})=
C_*(\T^{\pentagon}_{\partial M})\otimes\Z^{n-1},\qquad 
C_*(\T^\Delta_{\partial M};\Z^{n-1})=C_*(\T^\Delta_{\partial M})\otimes\Z^{n-1}.
\end{equation}

The intersection form induces a non-degenerate pairing
\begin{equation}
\omega\colon C_1(\T^{\pentagon}_{\partial M};\Z^{n-1})\times 
C_1(\T^{\pentagon}_{\partial M};\Z^{n-1})\to \Z,\quad 
(\alpha\otimes v,\beta\otimes w)\mapsto \iota(\alpha,\beta)\langle v,w\rangle,
\end{equation}
where $\langle,\rangle$ is the standard inner product on $\Z^{n-1}$. This 
pairing induces a pairing
\begin{equation}
\omega\colon H_1(\partial M;\Z^{n-1})\times H_1(\partial M;\Z^{n-1})\to \Z.
\end{equation}

\subsection{Definition of $\delta$}

Define
\begin{equation}
\delta\colon C_1(\T^{\pentagon}_{\partial M};\Z^{n-1})\to J^{\fg}(\T),
\qquad e^{ij}_\Delta\otimes e_r\mapsto \sum_{t_i=r}\sum_{s+e=t}t_j(s,e)_\Delta.
\end{equation}

\begin{figure}[htb]
\scalebox{0.8}{\input{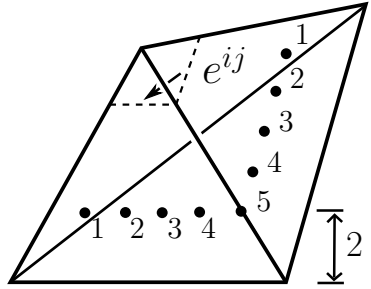}}
\caption{$\delta(e^{ij}\otimes e_2)$ for $n=7$. Each dot represents an 
integral point $t$ contributing a term $\sum_{s+e=t}(s,e)$. Interior terms 
are not shown, c.f.~Remark~\ref{rm:InteriorIgnore}.}
\label{fig:EdgeVertexFace2}
\end{figure}

Note that $\delta$ preserves rotational symmetry, i.e.~it is a map of 
$\Z[A_4]$-modules, where $A_4$ acts trivially on $\Z^{n-1}$.
\begin{proposition}
The map $\delta$ induces a map
\begin{equation}
\delta\colon H_1(\partial M;\Z^{n-1})\to H_3(\J^{\fg}).
\end{equation}
\end{proposition}

\begin{proof}
The result will follow by proving that there is a commutative diagram
\begin{equation}
\cxymatrix{{C_2(\T^{\pentagon}_{\partial M};\Z^{n-1})\ar[r]^-
\partial\ar[d]^{\delta_2}&C_1(\T^{\pentagon}_{\partial M};\Z^{n-1})\ar[r]^-
\partial\ar[d]^\delta&C_0(\T^{\pentagon}_{\partial M};\Z^{n-1})\ar[d]^{\delta_0}\\
C^{\fg}_1(\T)\ar[r]^-\beta&J^{\fg}(\T)\ar[r]^-{\beta^*}&C^{\fg}_1(\T).}}
\end{equation}

Define $\delta_2$ by
\begin{equation}
\rho^{\{i_lj_l\}}\otimes e_r\mapsto \sum_{l=1}^m\sum_{t_{i_l}=r}t_{j_l}t_{\Delta_l},
\qquad \tau^i_\Delta\otimes e_r\mapsto \sum_{t_i=r}(n-r)t_\Delta 
\end{equation}
Commutativity of the lefthand square is obvious for polygonal edges, and for 
triangular edges it follows from
\begin{equation}
\begin{aligned}
\delta\circ\partial(\tau^i_\Delta\otimes e_r)=&
\sum_{j\neq i}\delta(e^{ij}_\Delta\otimes e_r)\\
=&\sum_{t_i=r}\sum_{e+s=t}\sum_{j\neq i}t_j(e,s)\\
=&\sum_{t_i=r}\sum_{e+s=t}(n-t_i)(e,s)\\
=&\beta\circ\delta_2(\tau^i_\Delta\otimes e_r).
\end{aligned}
\end{equation}
Note that $\beta^*\circ\delta(e^{ij}\otimes e_r)
=\sum_{t_i=r}t_j\beta^*(\sum_{s+e}(s,e))$, which is a sum of hexagon relations 
(interior terms cancel). These involve only points on the faces determined 
by the start and end point of $e$, proving the existence of $\delta_0$. 
\end{proof}

\begin{remark}
Bergeron, Falbel and Guilloux~\cite{BFG} consider a map 
$H_3(\partial M;\Z^2)\to H_3(\J^{\fg})$ defined when $n=3$ using a different, 
but isomorphic chain complex. One can show that their map equals 
$\delta\circ\big(\id\otimes\left(\begin{smallmatrix}1&0\\0&2\end{smallmatrix}
\right)\big)$.
\end{remark}

\begin{remark}
\label{rm:InteriorIgnore}
In the formula for $\delta$ interior points may be ignored. This is because 
if $t$ is an interior point, then $\sum_{s+e=t}t_j(s,e)=t_j\beta(t)$.
\end{remark}

\subsection{Definition of $\gamma$}

The group $A_4$ acts transitively on the set of pairs of opposite edges of 
a simplex with stabilizer
\begin{equation}
D_4=\langle\id,(01)(23),(02)(13),(03)(12)\rangle\subset A_4.
\end{equation}
Hence, there is a one-one correspondence between $D_4$-cosets in $A_4$ and 
pairs of opposite edges. Explicitly,
\begin{equation}
\begin{gathered}
\Phi\colon A_4\big/D_4\mapsto\big\{\{\varepsilon_{01},\varepsilon_{23}\},
\{\varepsilon_{12},\varepsilon_{03}\},\{\varepsilon_{02},\varepsilon_{13}\}\big\}
\\
D_4\mapsto\{\varepsilon_{01},\varepsilon_{23}\},\qquad 
(012)D_4\mapsto \{\varepsilon_{12},\varepsilon_{03}\},\qquad 
(021)D_4\mapsto \{\varepsilon_{02},\varepsilon_{13}\}.
\end{gathered}
\end{equation}
Let $\bar e$ denote the opposite edge of $e$. Consider the map
\begin{equation}
\begin{gathered}
\gamma\colon J^{\fg}(\T)\to C_1(\T^\Delta_{\partial M};\Z^{n-1}) \\
(s,e)\mapsto\sum_{\sigma\in\Phi^{-1}(\{e,\bar e\})}E^{\sigma(1)\sigma(2)\sigma(3)}
\otimes v_{s,\sigma(1)},\qquad v_{s,i}=e_{s_i+1}-e_{s_i}
\end{gathered}
\end{equation}

The map $\gamma$ is illustrated in Figures~\ref{fig:Gamma01}, 
\ref{fig:Gamma12}, and \ref{fig:Gamma02}. For example, we have 
\begin{equation}
\gamma(s,\varepsilon_{01})=\gamma(s,\varepsilon_{23})
=E^{032}\otimes v_{s,0}+E^{123}\otimes v_{s,1}+E^{210}\otimes v_{s,2}
+E^{301}\otimes v_{s,3}.
\end{equation}


\begin{figure}[htb]
\begin{center}
\begin{minipage}[b]{0.3\textwidth}
\begin{center}
\scalebox{0.77}{\input{Gamma01.tex}}
\end{center}
\end{minipage}
\begin{minipage}[b]{0.3\textwidth}
\begin{center}
\scalebox{0.77}{\input{Gamma12.tex}}
\end{center}
\end{minipage}
\begin{minipage}[b]{0.3\textwidth}
\begin{center}
\scalebox{0.77}{\input{Gamma02.tex}}
\end{center}
\end{minipage}
\\
\begin{minipage}[t]{0.33\textwidth}
\begin{center}
\caption{$\gamma(s,\varepsilon_{01})$.}\label{fig:Gamma01}
\end{center}
\end{minipage}
\hspace{-3mm}
\begin{minipage}[t]{0.33\textwidth}
\begin{center}
\caption{$\gamma(s,\varepsilon_{12})$.}\label{fig:Gamma12}
\end{center}
\end{minipage}
\hspace{-3mm}
\begin{minipage}[t]{0.33\textwidth}
\begin{center}
\caption{$\gamma(s,\varepsilon_{02})$.}\label{fig:Gamma02}
\end{center}
\end{minipage}
\end{center}
\end{figure}

To see that $\gamma$ is well defined, note that 
$(s,\varepsilon_{01})+(s,\varepsilon_{12})+(s,\varepsilon_{02})$ maps to the 
boundary of $\sum_{i=0}^3T^i\otimes v_{s,i}$.

\begin{lemma}
$\gamma$ takes cycles to cycles and boundaries to boundaries.
\end{lemma}

\begin{proof}
We wish to show that $\gamma$ fits in a commutative diagram
\begin{equation}
\cxymatrix{{C^{\fg}_1(\T)\ar[r]^-\beta\ar[d]^-{\gamma_2}&
J^{\fg}(\T)\ar[r]^-{\beta^*}\ar[d]^-\gamma&
C^{\fg}_1(\T)\ar[d]^-{\gamma_0}\\
C_2(\T^{\Delta}_{\partial M};\Z^{n-1})\ar[r]^-{\partial}&
C_1(\T^{\Delta}_{\partial M};\Z^{n-1})\ar[r]^-\partial&
C_0(\T^{\Delta}_{\partial M};\Z^{n-1}),}}
\end{equation}
where $\gamma_2$ and $\gamma_0$ are defined by
\begin{equation}
\gamma_2(p)=
\begin{cases}
\displaystyle{\sum_{(t,\Delta)\in p}\,
\sum_{i\vert t_i>0} T^i_\Delta\otimes(e_{t_i}-e_{t_i-1})}&
p=\text{edge point}\\
\displaystyle{\sum_{(t,\Delta)\in p}\sum_i T^i_\Delta\otimes(e_{t_i}-e_{t_i-1})}
&p=\text{face point}\\
\displaystyle{\sum_{(t,\Delta)\in p}\sum_i T^i_\Delta\otimes(e_{t_{i+1}}-e_{t_i-1})}
&p=\text{interior point}
\end{cases}, 
\quad \gamma_0(p)=-\sum_{i,j}t_jV^{ij}_\Delta\otimes e_{t_i}.
\end{equation}
In the formula for $\gamma_0$, $(t,\Delta)$ is any representative of $p$. 
Commutativity of the lefthand side is clear from the geometry, and is shown 
for edge points in Figure~\ref{fig:Gamma2Edge}. To prove commutativity of 
the righthand side it is by rotational symmetry enough to consider 
$(s,\varepsilon_{01})$.
We have
\begin{multline}
\label{eq:downleft}
\partial\circ\gamma(s,\varepsilon_{01})
=(V^{02}-V^{03})\otimes(e_{s_0+1}-e_{s_0})+(V^{13}-V^{12})\otimes(e_{s_1+1}-e_{s_1})+
\\(V^{13}-V^{12})\otimes(e_{s_2+1}-e_{s_2})+(V^{31}-V^{30})\otimes(e_{s_3+1}-e_{s_3}).
\end{multline}
When expanding $\gamma_0\circ\beta^*(s,\varepsilon_{01})$, one gets a sum of 
$12$ (possibly vanishing) terms of the form $C_{ij}V^{ij}\otimes w_{ij}$, 
where $C_{ij}\in\Z$, $w_{ij}\in\Z^{n-1}$, and one must check that the terms 
agree with~\eqref{eq:downleft} (for example, we must have $C_{03}=1$, 
$w_{03}=e_{s_0+1}-e_{s_0}$).
We check this for the terms involving $V^{01}$ and $V^{02}$, and leave the 
verification of the other terms to the reader. Since, 
$\beta^*(s,\varepsilon_{01})=[s+\varepsilon_{03}]+[s+\varepsilon_{12}]
-[s+\varepsilon_{02}]-[s+\varepsilon_{13}]$, the term of 
$\gamma_0\circ\beta^*(s,e)$ involving $V^{01}$ equals
\begin{equation}
s_1V^{01}\otimes e_{s_0+1}+(s_1+1)V^{01}\otimes 
e_{s_0}-s_1V^{01}\otimes e_{s_0+1}-(s_1+1)V^{01}\otimes e_{s_0}=0.
\end{equation} 
Similarly, the term involving $V^{02}$ equals
\begin{equation}
-s_2V^{02}\otimes e_{s_0+1}-(s_2+1)V^{02}e_{s_0}+(s_2+1)V^{02}\otimes 
e_{s_0+1}+s_2V^{02}\otimes e_{s_0}=V^{02}\otimes(e_{s_0+1}-e_{s_0}).
\end{equation}
This proves the result.
\end{proof}

\begin{figure}
\begin{center}
\scalebox{0.4}{\input{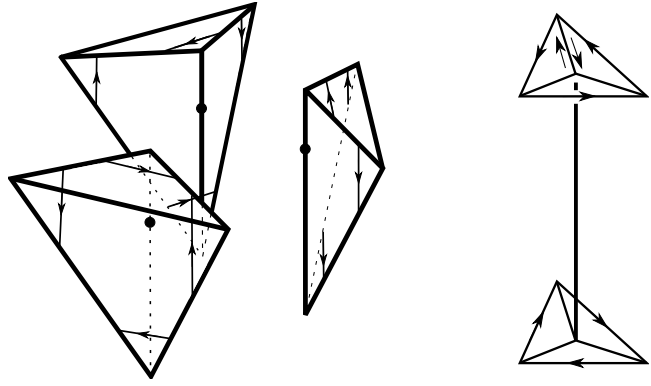}}
\caption{$\beta\circ\gamma(p)$ and $\partial\circ\gamma_2(p)$ for an 
edge point $p$.}
\label{fig:Gamma2Edge}
\end{center}
\end{figure}
Hence, we have
\begin{equation}
\gamma\colon H_3(J^{\fg})\to H_1(\partial M;\Z^{n-1}).
\end{equation}



\begin{proposition}
\label{prop:Adjoint}
The maps $\delta$ and $\gamma$ are adjoint, i.e.~we have
\begin{equation}
\Omega\big(\delta(e^{ij}\otimes e_r),(s,e)\big)
=\omega\big(e^{ij}\otimes e_r,\gamma(s,e)\big).
\end{equation}
\end{proposition}

\begin{proof}
By rotational symmetry it is enough to prove this for $e=\varepsilon_{01}$. 
We have
\begin{equation}
\Omega\big(\delta(e^{ij}\otimes e_r),(s,\varepsilon_{01})\big)
=\Omega\big(\sum_{t_i=r}
\sum_{s+\varepsilon=t}t_j(s,\varepsilon),(s,\varepsilon_{01})\big).
\end{equation}
Since $\Omega\big((s^\prime,e^\prime),(s,e)\big)=0$ when $s\neq s^\prime$, 
it follows that 
\begin{equation}
\Omega\big(\delta(e^{ij}\otimes e_r),(s,\varepsilon_{01})\big)=
\begin{cases}
\Omega\big(\displaystyle{\sum_{\varepsilon_i=0}}
(s_j+\varepsilon_j)(s,\varepsilon),(s,\varepsilon_{01})\big)&s_i=r\\
\Omega\big(\displaystyle{\sum_{\varepsilon_i=1}}
(s_j+\varepsilon_j)(s,\varepsilon),(s,\varepsilon_{01})\big)&s_i=r-1\\
0&\text{otherwise.}
\end{cases}
\end{equation}
Letting $f(i,j)=\langle \varepsilon_{ij},\varepsilon_{01}\rangle$, an 
inspection of Figure~\ref{fig:SL2Quiver} shows that
\begin{equation}
\Omega\big(\delta(e^{ij}\otimes e_r),(s,\varepsilon_{01})\big)=
\begin{cases}
-f(i,j)(-1)^{r-s_i}&\text{if }s_i=r\text{ or }s_i=r-1\\
0&\text{otherwise.}
\end{cases}
\end{equation}
The fact that this equals $\omega\big(e^{ij}\otimes e_r,
\gamma(s,\varepsilon_{01})\big)$ follows from \eqref{eq:IntersectionPairing} 
using Figures~\ref{fig:Gamma01}, \ref{fig:Gamma12} and~\ref{fig:Gamma02}.
\end{proof}

It will be convenient to rewrite the formula for $\delta$.

\begin{lemma}
We have
\begin{equation}
\delta(e^{ij}\otimes e_r)=\sum_{s_i=r-1}(s,\varepsilon_{ij})
-\sum_{s_i=r}(s,\varepsilon_{kl}),
\end{equation} 
where $k$ and $l$ are such that $\{i,j,k,l\}=\{0,1,2,3\}$.
\end{lemma}

\begin{proof}
By rotational symmetry, we may assume that $i=1$ and $j=0$. Using the 
relations $(s,\varepsilon_{01})+(s,\varepsilon_{12})+(s,\varepsilon_{02})=0$ 
and $(s,e)+(s,\bar e)=0$, we have
\begin{equation}
\begin{aligned}
\delta(e^{10}\otimes e_r)=&\sum_{t_1=r}\sum_{s+e=t}t_0(s,e)\\
=&\sum_{s_1=r-1}(s_0+1)(s,\varepsilon_{01})+\sum_{s_1=r}s_0(s,\varepsilon_{23})+
\\ &
\sum_{s_1=r-1}s_0(s,\varepsilon_{12})+\sum_{s_1=r}(s_0+1)(s,\varepsilon_{03})+
\\ &
\sum_{s_1=r-1}s_0(s,\varepsilon_{13})+\sum_{s_1=r}(s_0+1)(s,\varepsilon_{02})
\\=&
\sum_{s_i=r-1}(s,\varepsilon_{01})-\sum_{s_i=r}(s,\varepsilon_{23}).
\end{aligned}
\end{equation}
This proves the result.
\end{proof}

\begin{lemma}
\label{lemma:NearFar} 
Let $D=\diag\left(\{n-i\}_{i=1}^{n-1}\right)$ and $A_\fg$ denote the Cartan 
matrix of $\fg$.
\begin{equation}
\gamma\circ\delta(e^{ij}\otimes e_r)=E^{ikl}\otimes(\frac{1}{2}DA_\fg De_r)+
\\\big(E^{jlk}+E^{kij}+E^{lji}\big)\otimes e_{n-r}.
\end{equation}
where $k$ and $l$ are such that the permutation taking $ijkl$ to $0123$ is 
positive.
\end{lemma}

\begin{proof}
May assume that $i=1$ and $j=0$. Then $k=2$ and $l=3$. 
One thus has
\begin{equation}
\begin{aligned}
\label{eq:GammaDelta}
\gamma\circ\delta(e^{10}\otimes e_r)=&
\sum_{s_1=r-1}\gamma(s,\varepsilon_{01})-\sum_{s_0=r}\gamma(s,\varepsilon_{23})
\\ =&
\sum_{s_1=r-1}E^{123}\otimes(e_r-e_{r-1})-\sum_{s_1=r}E^{123}\otimes(e_{r+1}-e_r)+
\\ &
\sum_{s_1=r-1}E^{032}\otimes(e_{s_0+1}-e_{s_0})-\sum_{s_1=r}E^{032}\otimes
(e_{s_0+1}-e_{s_0})+\\ &
\sum_{s_1=r-1}E^{210}\otimes(e_{s_0+1}-e_{s_0})-\sum_{s_1=r}E^{210}\otimes
(e_{s_0+1}-e_{s_0})+
\\ &
\sum_{s_1=r-1}E^{301}\otimes(e_{s_0+1}-e_{s_0})-\sum_{s_1=r}E^{301}\otimes
(e_{s_0+1}-e_{s_0}).
\end{aligned}
\end{equation}
The number of subsimplices with $s_1=c$ equals $\frac{1}{2}(n-c)(n-c-1)$. We 
thus have
\begin{multline}
\label{eq:Near}
\sum_{s_1=r-1}(e_r-e_{r-1})-\sum_{s_1=r}(e_{r+1}-e_r)=\\
-\frac{1}{2}(n-r+1)(n-r)e_{r-1}+(n-r)^2e_r-\frac{1}{2}(n-r)(n-r-1)e_{r+1}
=\frac{1}{2}DA_\fg De_r.
\end{multline}
By telescoping, we have ($i$ and $j$ now general, not 1 and 0))
\begin{multline}
\label{eq:Far}
\sum_{s_1=r-1}E^{ijk}\otimes(e_{s_0+1}-e_{s_0})
-\sum_{s_1=r}E^{ijk}\otimes(e_{s_0+1}-e_{s_0})=\\
E^{ijk}\otimes\sum_{s_0=0}^{n-1-r}(e_{s_0+1}-e_{s_0})=E^{ijk}\otimes e_{n-r}.
\end{multline}
Plugging~\eqref{eq:Near} and \eqref{eq:Far} into \eqref{eq:GammaDelta} we 
end up with
\begin{equation}
\gamma\circ\delta(e^{10}\otimes e_r)=E^{123}\otimes\frac{1}{2}DA_\fg De_r
+E^{032}\otimes e_{n-r}+E^{210}\otimes e_{n-r}+E^{301}\otimes e_{n-r},
\end{equation}  
which proves the result.
\end{proof}

\begin{proposition}
\label{prop:NearFar}
The composition $\gamma\circ\delta\colon H_1(\partial M;\Z^{n-1})\to 
H_1(\partial M;\Z^{n-1})$ is given by
\begin{equation}
\gamma\circ\delta = \id\otimes DA_\fg D.
\end{equation} 
\end{proposition}

\begin{proof}
Let $\alpha=\sum a_m e^{i_mj_m}_{\Delta_m}$ be a cycle in 
$C_1(\T^{\pentagon}_{\partial M})$. In the proof of \cite[Lemma~4.3]{NeumannComb} 
(see also Bergeron--Falbel--Guilloux~\cite[Figures 12,13]{BFG}), 
Neumann proves that the ``near contribution'' 
\begin{equation}
\sum a_m E^{i_mk_ml_m}
\end{equation} 
is homologous to $2\alpha$, whereas the ``far contribution'' 
\begin{equation}
\sum a_m\big(E^{j_ml_mk_m}+E^{k_mi_mj_m}+E^{l_mj_mi_m}\big)
\end{equation}
is null-homologous. The result now follows from Lemma~\ref{lemma:NearFar}.
\end{proof}

\begin{proposition}
\label{prop:SameRank}
The groups $H_3(\J^{\fg})$ and $H_1(\partial M;\Z^{n-1})$ have equal rank.
\end{proposition}

\begin{proof}
Since all the outer homology groups have rank $0$, the rank of $H_3(\J)$ is 
the Euler characteristic $\chi( \J)$ of $\J$. Let $v$, $e$, $f$, and $t$ 
denote the number of vertices, edges, faces and tetrahedra, respectively, 
of $\T$. By a simple counting argument we have
\begin{equation}
\begin{gathered}
\rank (C_0^{\fg}(\T))=(n-1)v,\quad \rank(\J^{\fg}(\T))=2\binom{n+1}{3},\\
\rank(C_1^{\fg}(\T))=(n-1)e+\frac{(n-1)(n-2)}{2}f+\frac{(n-1)(n-2)(n-3)}{6}t.\\
\end{gathered}
\end{equation} 
Using the fact that $f=2t$ we obtain
\begin{multline}
\chi(\J)=2\rank(C_0^{\fg}(\T))-2\rank(C_1^{\fg}(\T))+\rank(J^{\fg}(\T))=\\
2(n-1)(v-e+t)=2(n-1)(v-e+f-t)=2(n-1)\chi(\widehat{M}).
\end{multline}
The result now follows from the elementary fact (proved by an Euler 
characteristic count) that $\chi(\widehat{M})=1/2\rank(H_1(\partial M))$. 
\end{proof}

\begin{corollary}
The groups $H_3(\J^{\fg})$ and $H_1(\partial M;\Z^{n-1})$ are isomorphic modulo 
torsion.\qed
\end{corollary}

\subsection{Proof of Theorem~\ref{thm:1}}

We now conclude the proof of Theorem~\ref{thm:1}. All that remains are the 
statements about the free part of $H_3(\J^{\fg})$.
We first show that $\gamma$ and $\delta$ admit factorizations
\begin{equation}
\begin{gathered}
\xymatrix{\delta\colon H_1(\partial M;\Z^{n-1})\ar[r]^-{\id\otimes D}&
H_1(\partial M;\Z^{n-1})\ar[r]^-{\delta^{\prime}}&H_3(\J^{\fg})}\\
\xymatrix{\gamma\colon H_3(\J^{\fg})\ar[r]^-{\gamma^{\prime}}&
H_1(\partial M;\Z^{n-1})\ar[r]^-{\id\otimes D}&H_1(\partial M;\Z^{n-1}).}
\end{gathered}
\end{equation}
The factorization of $\delta$ is constructed in the next section (see 
Proposition~\ref{prop:MuAndDelta}), and the factorization of $\gamma$ thus 
follows from Proposition~\ref{prop:Adjoint}. By 
Proposition~\ref{prop:NearFar}, we thus have 
\begin{equation}
\label{eq:GammaPrimeDeltaPrime}
\gamma^{\prime}\circ\delta^{\prime}=\id\otimes A.
\end{equation}
Since $\det(\A_{\fg})=n$, it follows that $\gamma^{\prime}$ maps onto a 
subgroup of $H_1(\partial M;\Z^{n-1})$ of index $h^n$. 
This shows that $\delta^{\prime}$ induces an isomorphism 
$H_1(\partial M;\Z[1/n]^{n-1})\to H_3(\J^{\fg})\otimes\Z[1/n]$ with inverse 
$\gamma^{\prime}$. The fact that $\Omega$ corresponds to $\omega_{A_\fg}$ 
follows from
\begin{equation}
\label{eq:OmegaAComputation}
\omega_{A_\fg}(\alpha\otimes v,\beta\otimes w)
=\omega(\alpha\otimes v,\beta\otimes Aw)
=\omega(\alpha\otimes v,\gamma^{\prime}\circ\delta^{\prime}(\beta\otimes w))
=\Omega\big(\delta^{\prime}(\alpha\otimes v),\delta^{\prime}(\beta\otimes w)\big).
\end{equation}






\section{Cusp equations and rank}
\label{sec:CuspEquations}

We express the cusp equations in terms of yet another decomposition of 
$\partial M$. This decomposition was introduced in 
Garoufalidis--Goerner--Zickert~\cite{GaroufalidisGoernerZickert}, and is 
the induced decomposition on $\partial M$ induced by the decomposition of 
$M$ obtained by truncating both vertices and edges. We call it the 
\emph{doubly truncated decomposition}. As in 
\cite{GaroufalidisGoernerZickert}, we label the edges by $\gamma^{ijk}$ and 
$\beta^{ijk}$. The superscript $ijk$ of an edge indicates the initial vertex 
(denoted by $v^{ijk}$) of the edge, $i$ being the nearest vertex of 
$\Delta$, $ij$, the nearest edge and $ijk$ the nearest face. As in 
Section~\ref{sec:LabelingConventions}, we label the hexagonal faces by 
$\tau^i$ and the polygonal faces by $p^{\{i_l,j_l\}}$.
\begin{figure}[htb]
\begin{center}
\begin{minipage}[b]{0.47\textwidth}
\begin{center}
\scalebox{0.65}{\input{HexagonDecomposition.tex}}
\end{center}
\end{minipage}
\hfill
\begin{minipage}[b]{0.47\textwidth}
\begin{center}
\scalebox{0.85}{\input{DoublyTruncatedCocycle.tex}}
\end{center}
\end{minipage}\\
\begin{minipage}[t]{0.47\textwidth}
\begin{center}
\caption{Doubly truncated decomposition of $\partial M$.}
\end{center}
\end{minipage}
\hfill
\begin{minipage}[t]{0.47\textwidth}
\begin{center}
\caption{Labeling conventions.}
\end{center}
\end{minipage}
\end{center}
\end{figure}

\subsection{Cusp equations}

For a shape assignment $z$ consider the map
\begin{equation}
\begin{gathered}
C(z)\colon C_1(\T^{\hexagon}_{\partial M};\Z^{n-1})\to \C^*,\\
\gamma^{ijk}\otimes e_r\mapsto 
(z^{\varepsilon_{ij}}_{(r-1)v_i+(n-r-1)v_j})^{-\varepsilon^{ijk}_\circlearrowleft},
\qquad \beta^{ijk}\mapsto\prod_{\substack{t\in\text{face}(ijk)\\t_i=r}}\!\!
(X_t)^{\varepsilon^{ijk}_\circlearrowleft},
\end{gathered}
\end{equation}
where $\varepsilon^{ijk}_\circlearrowleft$ is the sign of the permutation 
taking $ijkl$ to $0123$.
It follows from~\cite[Section~13]{GaroufalidisGoernerZickert} that $C(z)$ 
is a cocycle (it is the ratio of consecutive diagonal entries in the 
\emph{natural cocycle}~\cite{GaroufalidisGoernerZickert} associated to $z$). 
Hence, $C(z)$ may be regarded as a cohomology class 
$C(z)\in H^1(\partial M;(\C^*)^{n-1})$. This class vanishes if an only if 
for each representative $\alpha\in C_1(\T^{\hexagon}_{\partial M})$ of each 
generator of $H_1(\partial M)$, we have
\begin{equation}
\label{eq:CuspEqLevelr}
C(z)(\alpha\otimes e_r)=1.
\end{equation} 
This discussion is summarized in the result below.
\begin{theorem}
[Garoufalidis--Goerner--Zickert~\cite{GaroufalidisGoernerZickert}] 
The $\PGL(n,\C)$-representation determined by a shape assignment $z$ is 
boundary-unipotent if and only if all cusp 
equations are satisfied. Equivalently, if and only if $C(z)$ is trivial in 
$H^1(\partial M;(\C^*)^{n-1})$.\qed
\end{theorem}

The cusp equation~\eqref{eq:CuspEqLevelr} for $\alpha\otimes e_r$ can be 
written in the form
\begin{equation}
\prod_{s,\Delta}z_{s,\Delta}^{A^{\cusp}_{\alpha\otimes e_r,(s,\Delta)}}
\prod_{s,\Delta}(1-z_{s,\Delta})^{B^{\cusp}_{\alpha\otimes e_r,(s,\Delta)}}=\pm 1.
\end{equation}

\subsection{Linearizing the cusp equations}

Consider the map
\begin{equation}
\begin{gathered}
\delta^{\prime}\colon C_1(\T^{\hexagon}_{\partial M};\Z^{n-1})\to J^{\fg}(\T)\\
\gamma^{ijk}\otimes e_r\mapsto 
-\varepsilon^{ijk}_\circlearrowleft(r v_i+(n-r)v_j,\varepsilon_{ij}),\qquad 
\beta^{ijk}\otimes e_r\mapsto \varepsilon^{ijk}_\circlearrowleft
\sum_{\substack{t\in\text{face}(ijk)\\t_i=r}}\sum_{(s,e)=t}(s,e)
\end{gathered}
\end{equation} 
We wish to prove that $\delta^{\prime}$ induces a map in homology.
\begin{lemma}\label{lemma:Stokes}
For each $r=1,\dots,n-1$, we have
\begin{equation}\label{eq:Stokes}
\sum_{t\in\mathring{\Delta}^3_n(\Z),t_i=r}\beta(t)
=-\sum_{t\in\partial\Delta^3_n(\Z),t_i=r}\sum_{s+e=t}(s,e).
\end{equation}
\end{lemma}

\begin{proof}
Consider a slice of $\Delta_n^3(\Z)$ consisting of integral points with 
$t_i=r$ as shown in Figure~\ref{fig:Stokes} for $n-r=4$. Each dot represents 
an integral point $t$ and each vertex of each triangle intersecting $t$ 
represents an edge $e$ of a subsimplex $s$ with $s+e=t$. By~\eqref{eq:first} 
and \eqref{eq:second} the sum of the vertices (regarded as pairs $(s,e)$) of 
each triangle is zero. Using this it easily follows that the sum of all 
interior edges equals minus the sum of the boundary edges. 
Figure~\ref{fig:Stokes} shows the proof for $n-r=4$.
\end{proof}

\begin{figure}[htb]
\begin{center}
\scalebox{0.45}{\input{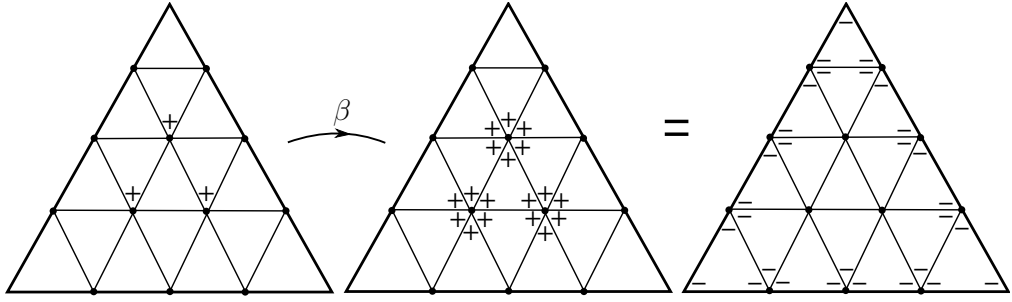}}
\caption{Proof of Lemma~\ref{lemma:Stokes} for $n-r=4$.}\label{fig:Stokes}
\end{center}
\end{figure}

\begin{proposition}
The map $\delta^{\prime}$ induces a map on homology.
\end{proposition}

\begin{proof}
We wish to extend $\delta^{\prime}$ to a commutative diagram
\begin{equation}
\cxymatrix{{C_2(\T^{\hexagon}_{\partial M};\Z^{n-1})\ar[r]^-
\partial\ar[d]^{\delta^{\prime}_2}&C_1(\T^{\hexagon}_{\partial M};
\Z^{n-1})\ar[r]^-\partial\ar[d]^\delta^{\prime}
&C_0(\T^{\hexagon}_{\partial M};\Z^{n-1})\ar[d]^{\delta^{\prime}_0}\\
C^{\fg}_1(\T)\ar[r]^-\beta&J^{\fg}(\T)\ar[r]^-{\beta^*}&C^{\fg}_1(\T).}}
\end{equation}

Define 
\begin{equation}
\begin{gathered}
\delta^{\prime}_2\colon C_2(\T^{\hexagon}_{\partial M};\Z^{n-1})\to C_1^{\fg}(\T)\\
\tau^i\otimes e_r\mapsto-\sum_{\substack{t\in\mathring{\Delta}^3_n(\Z),t_i=r}}t,
\qquad p^{i_lj_l}\mapsto\{(rv_{i_l}+(n-r)v_{j_l},\Delta_l)\}
\end{gathered}
\end{equation}
and 
\begin{equation}
\begin{gathered}
\delta^{\prime}_0\colon C_0(\T^{\hexagon}_{\partial M};\Z^{n-1})\to C_1^{\fg}(\T),\\
v^{ijk}\otimes e_r\mapsto \big((r+1)v_i+(n-r-1)v_j\big)-\big(rv_i+(n-r)v_j\big)
\\+\big((r-1)v_i+v_k+(n-r)v_j\big)-\big(rv_i+v_k+(n-r-1)v_j\big)
\end{gathered}
\end{equation}
The fact that $\beta\circ\delta^{\prime}_2(p^{i_lj_l}\otimes_r)
=\delta^{\prime}\circ\partial(p^{i_lj_l}\otimes e_r)$ is immediate, and the 
fact that $\beta\circ\delta^{\prime}_2(\tau^i\otimes e_r)=\delta^{\prime}
\circ\partial(\tau^i\otimes e_r)$ follows from Lemma~\ref{lemma:Stokes}. 
The terms involved in $\delta^{\prime}_0\circ\partial(\beta^{ijk}\otimes e_r)$ 
are the ones involved in the long hexagon relation shown in 
Figure~\ref{fig:MuZero}, and exactly corresponds to 
$\beta^*\circ\delta^{\prime}(\beta^{ijk}\otimes e_r)$. Finally, the equality 
$\beta^*\circ\delta^{\prime}(\gamma^{ijk}\otimes e_r)
=\delta^{\prime}_0\circ\partial(\gamma^{ijk}\otimes e_r)$ follows from the 
fact that the four edge terms of 
$\delta^{\prime}_0\circ\partial(\gamma^{ijk}\otimes e_r)$ cancel out, and the 
remaining $4$ terms are exactly those of 
$\beta^*\circ\delta^{\prime}(\gamma^{ijk}\otimes e_r)$.
\end{proof}

\begin{figure}[htb]
\begin{center}
\scalebox{0.85}{\input{MuZero.tex}}
\caption{$\delta^{\prime}_0\circ\partial(\beta^{120})$.}
\label{fig:MuZero}
\end{center}
\end{figure}

Let $z$ be a shape assignment on $\T$. Since 
$z_{s,\Delta}^{1100}z_{s,\Delta}^{0110}z_{s,\Delta}^{1010}=-1$ for each subsimplex 
$s$ of each simplex $\Delta$ of $\T$, it follows that $z$ defines an element 
$z\in \Hom(J^{\fg}(\T);\C^*\big/\{\pm 1\})$, and since the gluing equations are 
satisfied, we obtain an element $z\in H^3(\J^{\fg};\C^*\big/\{\pm 1\})$.

Dual to $\delta^{\prime}$ we have $\delta^{\prime*}\colon H^3(\J^{\fg};\C^*)\to 
H^1(\partial M;(\C^*)^{n-1})$. The following follows immediately from the 
definitions. 

\begin{proposition}
We have $\delta^{\prime*}(z)=C(z)\in H^1(\partial M;(\C^*\big/\{\pm 1\})^{n-1})$.
\qed
\end{proposition}

In particular, 
\begin{equation}
\delta^{\prime^*}(z)=\prod_{s,\Delta}z_{s,\Delta}^{A_{\alpha,(s,\Delta)}}
\prod_{s,\Delta}(1-z_{s,\Delta})^{B_{\alpha,(s,\Delta)}}.
\end{equation}

For any abelian group $A$, we shall use the canonical identifications
\begin{equation}
\Hom(H_1(\partial M;\Z^{n-1}),A)\cong\big(\Hom(H_1(\partial M),A)\big)^{n-1}
\cong H^1(\partial M;A^{n-1}).
\end{equation}
If $\phi$ is an element of $\Hom(H_1(\partial M;\Z^{n-1}),A)$ or 
$H^1(\partial M;A^{n-1})$, we let $\phi_r\colon H_1(\partial M)\to A$ 
denote the $r$th coordinate function. 

\begin{proposition}
\label{prop:MuAndDelta}
We have 
\begin{equation}
\delta=\delta^{\prime}\circ (\id\otimes D)\in 
\Hom\big(H_1(\partial M;\Z^{n-1}),H_3(\J^{\fg})\big),\qquad 
D=\diag(\{n-i\}_{i=1}^{n-1}).
\end{equation}
Equivalently, $\delta_r=(n-r)\delta^{\prime}_r$ for all $r=1,\dots,n-1$.
\end{proposition}

\begin{proof}
We prove the second statement. Every class in $H_1(\partial M)$ can be 
represented by a curve $\alpha$ which is a sequence of left and right 
turns as shown in Figure~\ref{fig:LeftRightCurve}. We can represent 
$\alpha$ in $C_1(\T^{\pentagon}_{\partial M})$ and $C_1(\T^{\hexagon}_{\partial M})$ 
as shown in Figure~\ref{fig:RepresentGamma}. Namely, the representation in 
$C_1(\T^{\pentagon}_{\partial M})$ is the natural one, and the representation 
in $C_1(\T^{\hexagon}_{\partial M})$ is obtained by replacing a left turn by 
its corresponding $\gamma$ edge, and a right turn by a concatenation 
$\beta\gamma\beta$.
The contribution to $\delta^{\prime}_r(\alpha)$ and $\delta_r(\alpha)$ from 
a left and right turn are shown schematically in Figures~\ref{fig:LeftDelta} 
and \ref{fig:RightDelta} (the interior points are ignored, 
c.f.~Remark~\ref{rm:InteriorIgnore}). Each dot represents an integral point 
$t$, contributing the terms $\sum_{s+e=t}(s,e)$. We wish to prove that 
$\delta_r(\alpha)=(n-r)\delta^{\prime}_r(\alpha)$, whenever $\alpha$ is a 
cycle. This can be seen by inspecting Figures~\ref{fig:LeftRight} and 
\ref{fig:RightRight}; recall that when two faces are paired the terms on 
each side differ by an element in the image of $\beta$.
\end{proof}

\begin{figure}[htb]
\begin{center}
\begin{minipage}[b]{0.47\textwidth}
\begin{center}
\scalebox{0.6}{\input{LeftRightCurve.tex}}
\end{center}
\end{minipage}
\hfill
\begin{minipage}[b]{0.47\textwidth}
\begin{center}
\scalebox{0.6}{\input{RepresentingGamma.tex}}
\end{center}
\end{minipage}\\
\begin{minipage}[t]{0.47\textwidth}
\begin{center}
\caption{Left and right turns.}\label{fig:LeftRightCurve}
\end{center}
\end{minipage}
\hfill
\begin{minipage}[t]{0.47\textwidth}
\begin{center}
\caption{Representing a curve in $C_1(\T^{\pentagon}_{\partial M})$ and 
$C_1(\T^{\hexagon}_{\partial M})$.}
\label{fig:RepresentGamma}
\end{center}
\end{minipage}
\end{center}
\end{figure}

\begin{figure}[htb]
\begin{center}
\begin{minipage}[b]{0.47\textwidth}
\begin{center}
\scalebox{0.39}{\input{LeftDelta.tex}}
\end{center}
\end{minipage}
\hfill
\begin{minipage}[b]{0.47\textwidth}
\begin{center}
\scalebox{0.39}{\input{RightDelta.tex}}
\end{center}
\end{minipage}\\
\begin{minipage}[t]{0.47\textwidth}
\begin{center}
\caption{$\delta$ and $\delta^{\prime}$ for a left turn.}
\label{fig:LeftDelta}
\end{center}
\end{minipage}
\hfill
\begin{minipage}[t]{0.47\textwidth}
\begin{center}
\caption{$\delta$ and $\delta^{\prime}$ for a right turn.}
\label{fig:RightDelta}
\end{center}
\end{minipage}
\end{center}
\end{figure}

\begin{figure}[htb]
\begin{center}
\begin{minipage}[b]{0.47\textwidth}
\begin{center}
\scalebox{0.6}{\input{LeftRight.tex}}
\end{center}
\end{minipage}
\hfill
\begin{minipage}[b]{0.47\textwidth}
\begin{center}
\scalebox{0.6}{\input{RightRight.tex}}
\end{center}
\end{minipage}\\
\begin{minipage}[t]{0.47\textwidth}
\begin{center}
\caption{$\delta$ and $\delta^{\prime}$ for a left turn followed by a 
right turn.}
\label{fig:LeftRight}
\end{center}
\end{minipage}
\hfill
\begin{minipage}[t]{0.47\textwidth}
\begin{center}
\caption{$\delta$ and $\delta^{\prime}$ for two right turns.}
\label{fig:RightRight}
\end{center}
\end{minipage}
\end{center}
\end{figure}

\subsection{Proof of Corollaries~\ref{cor:PoissonCommuteGeneral} 
and \ref{cor:RankGeneral}}

By comparing the generalized gluing equation~\eqref{eq:DefGeneralizedGluingEq} 
with the definition~\eqref{eq:DefBeta} of $\beta$ we obtain that 
\begin{equation}
\beta(p)=\sum_{(s,\Delta)} A_{p,(s,\Delta)}(s,\varepsilon_{01})
+\sum_{s,\Delta} B_{p,(s,\Delta)}(s,\varepsilon_{12}).
\end{equation}
Also, by definition of $\delta^{\prime}$, we have
\begin{equation}
\delta^{\prime}(p)=
\sum_{(s,\Delta)} A^{\cusp}_{\alpha\otimes e_r,(s,\Delta)}(s,\varepsilon_{01})
+\sum_{s,\Delta} B^{\cusp}_{\alpha\otimes e_r,(s,\Delta)}(s,\varepsilon_{12}),
\end{equation}
and is in $\Ker(\beta^*)$. Since $\beta^*\circ\beta=0$, $\Ker(\beta^*)$ is 
orthogonal to $\Im(\beta)$ proving the first statement of 
Corollary~\ref{cor:PoissonCommuteGeneral}. The second statement follows 
from~\eqref{eq:OmegaAComputation}. Finally, 
Corollary~\ref{cor:RankGeneral} follows immediately from the fact that 
$H_4(\J^{\fg})$ is zero modulo torsion.

\subsection*{Acknowledgment} The authors wish to thank Matthias 
Goerner and Walter Neumann for useful conversations.


\bibliographystyle{plain}
\bibliography{BibFile}

\end{document}